\documentclass[11pt,a4paper]{article}

\usepackage{epsf,epsfig,amsfonts,amsgen,amsmath,amstext,amsbsy,amsopn,amsthm}
\usepackage{amsmath}
\usepackage{amsfonts,amsthm,amssymb,bm}
\usepackage{amsfonts}
\usepackage{graphics}
\usepackage{latexsym,bm}
\usepackage{amsfonts,amsthm,amssymb,bbding}
\usepackage{indentfirst}
\usepackage{graphicx}
\usepackage{color}

\usepackage[colorlinks=true,anchorcolor=blue,filecolor=blue,linkcolor=blue,urlcolor=blue,citecolor=red]{hyperref}
\usepackage{float}
\usepackage{tikz,enumerate}

\usepackage{changes}   

\usepackage[margin=2.5cm]{geometry}

\newtheorem{thm}{Theorem}[section]
\newtheorem{prob}[thm]{Problem}

\newtheorem{lem}[thm]{Lemma}
\newtheorem{cor}[thm]{Corollary}

\newtheorem{claim}{Claim}[section]
\newtheorem{definition}{Definition}[section]

\addtocounter{section}{0}

\usepackage{enumitem}  

\begin{document}
\title{Edge-spectral supersaturation for tripartite color-critical graphs
\footnote{Supported by the National Natural Science Foundation of China (Nos.\,12501471,\,12271162,\,12571369),
and the Natural Science Foundation of Shanghai (No.\,22ZR1416300).}}
\author{{\bf Longfei Fang$^{a}$},~{\bf Huiqiu Lin$^{b}$},~
{\bf Mingqing Zhai$^c$}\thanks{Corresponding author: \texttt{mqzhai@njust.edu.cn}
(M. Zhai)} \\[2mm]
\small $^{a}$ School of Mathematics and Finance, Chuzhou University,\\
\small  Chuzhou, Anhui 239012, China\\
\small $^{b}$ School of Mathematics, East China University of Science and Technology, \\
\small  Shanghai 200237, China\\
\small $^{c}$ School of Mathematics and Statistics, Nanjing University of Science and
Technology, \\
\small Nanjing, Jiangsu 210094, China
}

\date{}
\maketitle

\begin{abstract}
A central problem in spectral supersaturation asks
whether an $m$-edge graph whose spectral radius exceeds the corresponding edge-spectral Tur\'{a}n threshold must contain not only one copy,
but polynomially many copies, of a prescribed color-critical graph.
While this problem has been studied systematically for color-critical graphs of chromatic number at least four,
the three-chromatic case is more delicate,
since the exact threshold is often governed by a graph-dependent split construction rather than by the leading term $\sqrt m$ alone.
A related problem for odd cycles asks for the sharp asymptotic value of the minimum normalized number of copies of $C_{2k+1}$ above the refined spectral threshold.

We study edge-spectral supersaturation for two families of color-critical graphs with chromatic number three.
For an integer $r\geq 1$,
we define the spectral threshold
\[
g_r(m):=\frac{r-1+\sqrt{4m-r^2+1}}{2},
\]
which is the tight upper bound on the spectral radius of graphs avoiding $K_{s,t}^+$
(when $t+1\geq s=r+1\geq 3$) and $C_{2k+1}$ (when $r=k$),
realized by split-graph constructions.
First, let $t+1 \geq s\geq 3$ be fixed integers, and let $K_{s,t}^{+}$ be obtained by adding an edge to the part of size $s$ in $K_{s,t}$.
We prove that every sufficiently large $m$-edge graph $G$ with $\rho(G)>g_{s-1}(m)$ contains $\Omega(m^{(s+t-1)/2})$ copies of $K_{s,t}^{+}$.
Second, for any fixed $k\geq 2$, the condition $\rho(G)>g_k(m)$ forces $N(C_{2k+1},G)=\Omega(m^k).$
We also construct graphs showing that both lower bounds are tight up to constant factors.
These results establish that exceeding the tight spectral Tur\'{a}n threshold $g_r(m)$ forces not just a single copy,
but the optimal polynomial number of copies of these color-critical graphs.
Thus, crossing the relevant split-graph spectral threshold forces the optimal polynomial order of copies,
extending edge-spectral existence theorems to supersaturation results in the delicate three-chromatic regime.
\end{abstract}

\begin{flushleft}
\end{flushleft}

\section{Introduction}\label{sec1A}

Supersaturation is a central theme in extremal graph theory.
Tur\'{a}n's theorem states that every $n$-vertex graph with more than $e(T_{n,r})$ edges contains a copy of $K_{r+1}$,
where $T_{n,r}$ is the complete $r$-partite graph whose parts are as equal as possible.
A natural quantitative refinement asks how many copies of a fixed graph must occur once the corresponding extremal threshold is exceeded.

The first result of this type was obtained by Rademacher in 1941,
who proved that every $n$-vertex graph with $e(T_{n,2})+1$ edges contains at least $\lfloor n/2\rfloor$ copies of $K_3$.
This line of research was subsequently extended from cliques to more general forbidden graphs.
Recall that a graph $F$ with $\chi(F)=r+1$ is \emph{color-critical} if deleting some edge of $F$ decreases its chromatic number.
Simonovits~\cite{Simonovits1968} proved that, for every color-critical graph $F$ and all sufficiently large $n$,
$T_{n,r}$ is the unique extremal $F$-free graph.
Mubayi~\cite{Mubayi2010} later established a unified counting refinement,
showing that exceeding the Tur\'{a}n threshold forces many copies of $F$.
Further developments on classical supersaturation can be found in~\cite{LS1983,MY2025,Nikiforov2011,PY2017}.

A spectral analogue of this problem is obtained
by replacing the edge-density condition with a condition on the adjacency spectral radius.
More precisely, one asks whether exceeding the spectral Turán threshold for $F$ forces not merely one copy,
but quantitatively many copies, of $F$.
The aim of this paper is to investigate this question in the edge-spectral setting,
where the number of edges,
rather than the number of vertices, is prescribed.

Throughout the paper, all graphs are finite and simple. For a graph $G$, we write $|G|$ for its number of vertices,
$e(G)$ for its number of edges, and $\rho(G)$ for the spectral radius of its adjacency matrix $A(G)$.
For a fixed graph $F$, let $N(F,G)$ denote the number of copies of $F$ in $G$.
Since adding isolated vertices affects neither the number of edges nor the spectral radius,
extremal graphs are assumed to have no isolated vertices unless otherwise stated.

\subsection{The Brualdi-Hoffman-Tur\'{a}n type problem}

A classical starting point is Nosal's theorem \cite{Nosal1970},
which asserts that every triangle-free graph $G$ with $m$ edges satisfies
$\rho(G)\leq \sqrt m$.
Equivalently, the condition $\rho(G)>\sqrt m$ forces the existence of a triangle.
Since the Rayleigh quotient gives
$\rho(G)\geq \frac{2e(G)}{|G|}$,
Nosal's theorem may be viewed as an edge-spectral strengthening of Mantel's theorem. More generally,
Nikiforov \cite{Nikiforov2002,Nikiforov2006,Nikiforov2009-JCTB} proved that every $K_{r+1}$-free graph $G$ with $m$ edges satisfies
$\rho^2(G)\leq \left(1-\frac1r\right)2m$.
This is the edge-spectral counterpart of Tur\'{a}n's theorem for cliques,
and it motivates the following Brualdi-Hoffman-Tur\'{a}n type problem:
determine the maximum possible spectral radius of an $F$-free graph with a prescribed number of edges.
Compared with vertex-spectral conditions, edge-spectral conditions are particularly flexible,
since they apply naturally to graphs of arbitrary edge density, including sparse graphs.

This problem also has a natural asymptotic form.
Recently, Li, Liu, and Zhang \cite{Li2025+} proved an edge-spectral Erd\H{o}s-Stone-Simonovits theorem:
if $\chi(F)=r+1\geq 3$ and $G$ is an $F$-free graph with $m$ edges, then
$\rho^2(G)\leq \left(1-\frac1r+o(1)\right)2m$,
where $o(1)\to 0$ as $m\to\infty$ for fixed $F$.
They also established a corresponding edge-spectral stability theorem,
showing that near-extremal graphs are close, in edit distance,
to complete bipartite graphs when $r=2$, and to $r$-partite Tur\'{a}n-type graphs when $r\geq 3$.

Beyond the clique case,
the exact extremal graph depends sensitively on the structure of the forbidden graph.
For any color-critical graph $F$ with $\chi(F)=r+1\ge 4$,
the Nikiforov bound holds for all sufficiently large $m$;
moreover, if equality is attained, then the extremal graph is necessarily a regular complete $r$-partite graph.
In contrast, for many forbidden graphs that are either bipartite or color-critical with chromatic number three,
the extremal graphs are typically of split type:
they consist of a bounded clique joined to a large independent set, possibly augmented by a single vertex of prescribed degree;
see \cite{Fang2026,Li2025+C}.

One relevant family of color-critical graphs with chromatic number three consists of chorded cycles.
For $\ell\geq 4$,
let $C_\ell^+$ denote the graph obtained from $C_\ell$ by adding an edge joining two vertices at distance two on the cycle.
Li, Zhai, and Shu~\cite{LZS2024} established a sharp edge-spectral theorem for $C_{2k+1}^+$-free and $C_{2k+2}^+$-free graphs.
Since $C_\ell$ is a subgraph of $C_\ell^+$, their result yields the following consequence for ordinary cycles.

\begin{thm}[\cite{LZS2024}]\label{thm:LZS-cycle}
Let $k\geq 3$ and $m\geq 4(k^2+3k+1)^2$. If $G$ is $C_{2k+1}$-free or $C_{2k+2}$-free, then
\[
   \rho(G)\leq \frac{k-1+\sqrt{4m-k^2+1}}{2}.
\]
Equality holds if and only if
$G\cong K_k\vee E_{\frac{m}{k}-\frac{k-1}{2}}$,
where $E_k$ denotes the $k$-vertex empty graph.
\end{thm}

For later reference, fix an integer $k\geq 1$ and define
\[
   g_k(m):=\frac{k-1+\sqrt{4m-k^2+1}}{2}.
\]
If $m-\binom{k}{2}$ is divisible by $k$,
then the spectral radius of $K_k\vee E_{\frac{m}{k}-\frac{k-1}{2}}$ is exactly $g_k(m)$.
Moreover, for every fixed $k$, we have
\begin{align}\label{equ-001}
   g_k(m)=\sqrt m+\frac{k-1}{2}+O(m^{-1/2}).
\end{align}
Thus, although the leading term is still $\sqrt m$,
the constant-order correction $(k-1)/2$ is decisive
for the corresponding exact edge-spectral extremal problem.

To treat all values of $m$, it is convenient to use the following split graph.

\begin{definition}[The split graph]
Let $s\geq 1$ and $m\geq \binom{s}{2}$. Write uniquely
\[
m-\binom{s}{2}=sq+r,
\qquad q\geq 0,\quad 0\leq r\leq s-1.
\]
If $r=0$, define
$S_{s,m}:=K_s\vee E_q$.
If $1\leq r\leq s-1$, let $S_{s,m}$ be the graph obtained from
$K_s\vee E_q$ by adding one additional vertex adjacent to exactly
$r$ vertices of $K_s$.
\end{definition}

Using the edge-spectral stability method,
Li, Liu, and Zhang \cite{Li2025+C} established an asymptotic formula and structural characterization
for bipartite graphs and tripartite color-critical graphs,
and in particular obtained exact results for $K_{s,t}^+$,
where $K_{s,t}^+$ denotes the graph obtained from the complete bipartite graph $K_{s,t}$
by adding one edge inside the part of size $s$.

\begin{thm}[\cite{Li2025+C}]\label{thm:Kplus}
Let $3\leq s\leq t$ be fixed and $m$ be sufficiently large.
If $G$ is a $K_{s,t}^+$-free graph with $m$ edges, then
\[
   \rho(G)\leq \rho(S_{s-1,m}),
\]
and equality holds if and only if $G\cong S_{s-1,m}$.
\end{thm}

The methodology proposed in \cite{Li2025+C} to prove Theorem~\ref{thm:Kplus} can be straightforwardly extended to the case where
$t+1=s\geq 3$ with no substantial adjustments required,
yielding the identical conclusion.
A direct calculation shows that
$\rho(S_{s-1,m})\leq g_{s-1}(m).$
Consequently, $g_{s-1}(m)$ provides a convenient uniform analytic threshold,
while $\rho(S_{s-1,m})$ records the exact dependence on the residue class of $m$.

These results illustrate a characteristic feature of edge-spectral extremal problems in the $\chi=3$ regime.
Although the leading term of the threshold is typically $\sqrt m$,
the exact extremal graph and the relevant lower-order correction depend on the structure of the forbidden graph.
In particular, split constructions arise naturally for odd cycles, chorded cycles, fan graphs, friendship graphs, theta graphs,
and other almost-bipartite graphs; see~\cite{LZZ25,Liu2026}.
This structural dependence distinguishes the case $\chi=3$ from the more uniform higher-chromatic setting.

\subsection{Supersaturation via spectral radius}

The study of spectral supersaturation was initiated by Bollob\'{a}s and Nikiforov \cite{Bollobas2007},
who obtained inequalities relating the number of cliques in a graph to its spectral radius.
In particular, their result for triangles implies that
$N(K_3,G)\geq \frac{1}{3}\rho(G)\bigl(\rho^2(G)-m\bigr).$
Ning and Zhai \cite{NZ2021} subsequently proved a sharp counting refinement at the critical threshold.
They showed that if $G$ has $m$ edges and $\rho(G)\geq \sqrt m$,
then, unless $G$ is a complete bipartite graph possibly together with isolated vertices,
$N(K_3,G)\geq \lfloor\frac{\sqrt m-1}{2}\rfloor.$
Moreover, this lower bound is sharp.
Thus, the edge-spectral threshold $\sqrt m$ not only guarantees the existence of a triangle,
but also determines the correct order of the minimum number of triangles.

More recent work has extended this perspective beyond individual triangles.
Li, Liu, and Zhang \cite{Li2026} proved that
every $m$-edge graph with $\rho(G)>\sqrt m$ contains $\Omega(\sqrt m)$ triangles sharing a common edge.
They also determined the asymptotic minimum number of $4$-cycles.
Specifically, they proved that any graph $G$ with $\rho(G) > \sqrt{m}$ contains at least
$(\frac{1}{8}-o(1)) m^2$
copies of $C_4$, where the constant $1/8$ is optimal.
These results suggest a broader principle: once the correct
edge-spectral Tur\'{a}n threshold for a graph \(F\) is crossed, not only one copy of \(F\),
but polynomially many copies of \(F\), should be forced.
Furthermore, Li, Liu, and Zhang proposed an intriguing problem as follows:

\begin{prob}[\cite{Li2026}]\label{ques0}
Investigate spectral supersaturation
for general color-critical graphs with chromatic number $r+1$ in an $m$-edge graph
with spectral radius greater than $\sqrt{\big(1-\frac{1}{r}\big)2m}$.
\end{prob}

For color-critical graphs of chromatic number at least four,
this philosophy has recently been confirmed in a rather general form.
Li, Liu, and Zhang \cite{Li2025+C} proved an edge-spectral Tur\'{a}n theorem for color-critical graphs:
if $F$ is color-critical with $\chi(F)=r+1\geq 4$,
then every sufficiently large $m$-edge $F$-free graph $G$ satisfies
$\rho^2(G)\leq \left(1-\frac1r\right)2m$,
with equality if and only if $G$ is a regular complete $r$-partite graph.
Subsequently, in \cite{Fang2025+A}, we established a corresponding
edge-spectral supersaturation and stability framework. In particular, we
proved that exceeding this spectral threshold forces the presence of the
asymptotically optimal number of copies of $F$. These results may be viewed
as edge-spectral counterparts to the classical Rademacher-Mubayi theory
for color-critical graphs.

The remaining case $\chi(F)=3$ is substantially more delicate.
In this regime, there is no single exact threshold determined solely by the chromatic number.
Although the leading term is often $\sqrt m$, lower-order terms depend on the structure of $F$,
and the extremal graphs are frequently split graphs rather than complete bipartite graphs.
Hence,
a spectral supersaturation theorem must take into account the precise split-graph threshold associated with the forbidden graph.

The two families considered in this paper provide natural test cases for this phenomenon.
Both $K_{s,t}^+$ and $C_{2k+1}$ are color-critical graphs of chromatic number three:
deleting the added internal edge from $K_{s,t}^+$, or deleting any edge from $C_{2k+1}$,
makes the graph bipartite.
Their edge-spectral extremal constructions are governed by split graphs,
but the corresponding counting problems have different combinatorial features.

For odd cycles, the natural supersaturation problem was formulated as follows.

\begin{prob}[\cite{Li2026+A}]\label{Prob1.2}
Determine the sharp asymptotic constant
\[
\inf_{\substack{e(G)=m\\ \rho(G)>g_k(m)}}
\frac{N(C_{2k+1},G)}{m^k}
\]
as $m \rightarrow \infty$.
\end{prob}

The analogous problem for $K_{s,t}^+$ asks for the minimum number of copies forced
when the spectral radius exceeds the corresponding split-graph threshold.
Our first main result determines the correct polynomial order in this setting.

\subsection{Main results}

The purpose of this paper is to establish edge-spectral supersaturation bounds
of the correct polynomial order for two natural families of color-critical
graphs with chromatic number three: the graphs $K_{s,t}^+$ and odd cycles.
Both $K_{s,t}^+$ and $C_{2k+1}$ are color-critical: deleting the added edge from $K_{s,t}^+$,
or deleting any edge from $C_{2k+1}$, makes the graph bipartite.
Our results therefore address the tripartite color-critical regime in which the extremal threshold is governed by split graphs.
Our first main theorem gives a counting version for
\(K_{s,t}^{+}\) at the threshold $g_{s-1}(m)$.

\begin{thm}\label{thm1.1}
Let $t$ and $s$ be integers satisfying $t+1\geq s\geq 3$.
For sufficiently large \(m\), if \(G\) is an \(m\)-edge graph satisfying
$\rho(G)> g_{s-1}(m),$
then $$N(K_{s,t}^+,G)=\Omega(m^{\frac{s+t-1}{2}}),$$
and this bound is tight up to a constant factor.
\end{thm}

While the case  $s=2$ corresponds to book graphs and lies outside the scope of Theorem~\ref{thm1.1},
this scenario has attracted substantial recent research attention (see \cite{C-L-T,Li2026,ZhaiLiLou,ZhaoYouZengZhang}).
From this perspective, Theorem~\ref{thm1.1} addresses the next regime of complete bipartite graphs plus an internal edge;
here the larger critical partition gives rise to a distinct counting problem.

Our second main result concerns odd cycles.
The case $C_3$ has already been well understood in the spectral supersaturation setting:
under the condition $\rho(G)>\sqrt m$,
the best possible number of triangles is $\lfloor\frac{\sqrt m-1}{2}\rfloor$ (see \cite{NZ2021}).
In contrast, longer odd cycles exhibit a different behavior.
For $C_{2k+1}$ with $k\geq 2$,
the refined threshold $g_k(m)$ forces a much larger number of copies.

\begin{thm}\label{thm1.2}
For every integer \(k\geq 2\) and all sufficiently large \(m\), if \(G\) is a graph with \(m\) edges such that
$\rho(G)>g_k(m),$
then $$N(C_{2k+1},G)=\Omega(m^{k}),$$
and this bound is tight up to a constant factor.
\end{thm}


Theorem~\ref{thm1.2} improves upon the corresponding edge-spectral extremal result for odd cycles
by strengthening its assertion from forcing one copy to the optimal number of copies,
thus offering a partial resolution to Problem~\ref{Prob1.2}.

\begin{cor}\label{cor1.1}
Let $k\geq 2$ be an integer, and let $F$ be a graph satisfying
$C_{2k+1}\subseteq F\subseteq K_{k+1,k}^{+}.$
Then, for all sufficiently large $m$, every graph $G$ with $m$ edges and
$\rho(G)>g_k(m)$
satisfies
$$
N(F,G)=\Omega(m^k).
$$
Moreover, this bound is tight up to a constant factor.
\end{cor}

The remainder of the paper is organized as follows. Section~\ref{sec2} introduces the necessary notation and auxiliary lemmas. Section~\ref{sec3} presents a construction showing that the bound in Theorem~\ref{thm1.1} is tight up to a constant factor. Section~\ref{sec4} develops the structural properties of the core subgraph, and Section~\ref{sec5A} completes the proof of Theorem~\ref{thm1.1}. Finally, Section~\ref{sec6} contains the proofs of Theorem~\ref{thm1.2} and Corollary~\ref{cor1.1}.

\section{Preliminaries}\label{sec2}

We start with three notations that will be used to formulate our stability statements.
Given two graphs $G$ and $H$ (which may have distinct vertex sets),
we define their \emph{distance} as
$$d(G, H):=|E(G)\setminus E(H)|+|E(H)\setminus E(G)|,$$
which counts the minimum number of edge modifications (additions or deletions) needed to transform $H$ into $G$.
Given $r$ disjoint vertex sets $U_1,\dots,U_r$,
we use $K_{U_1,\dots,U_r}$ to represent the complete $r$-partite graph with parts $U_1,\dots,U_r$.
For a bipartite graph $F$, we denote by $\beta'(F)$ the cardinality of a minimum independent set $V_0\subseteq V(F)$
such that every edge in $E(F)$ has at least one endpoint in $V_0$.

The following theorem establishes a
spectral supersaturation phenomenon, which guarantees that once the spectral
radius of an $m$-edge graph $G$ strictly exceeds the corresponding threshold with respect to $F$,
the number of copies of
$F$ in $G$ must grow at the scale of $m^{|F|/2}$.

\begin{lem}[\cite{Fang2025+A}]\label{thm2.2G}
Let $F$ be an $f$-vertex graph with $\chi(F)=r+1\geq2$.
For any $\varepsilon >0$,
there exists a constant $\delta=\delta(F,\varepsilon)>0$ such that
for every graph $G$ of sufficiently large size $m$ whose spectral radius satisfies
 $$\rho(G)\geq \left\{
                                       \begin{array}{ll}
                                         \sqrt{\big(1-\frac{1}{r}+\varepsilon\big)2m}  & \hbox{if $r\geq 2$,} \\
                                         \sqrt{\big(1+\varepsilon\big)m} & \hbox{if $r=1$ and $\beta'(F)\geq 2$,}
                                       \end{array}
                                     \right.
$$
we have $N(F,G)\geq \delta\cdot m^{f/2}$.
\end{lem}

To analyze the structural properties of our extremal graphs with respect to $N(F,G)$, we require the
following spectral supersaturation-stability result.

\begin{lem}[Edge-spectral supersaturation-stability \cite{Fang2025+A}]\label{thm2.1G}
Let $F$ be a fixed graph of order $f$ with $\chi(F)=r+1\geq2$,
and let $G$ be a graph of sufficiently large size $m$ with $N(F,G)=o(m^{f/2})$.
For every $\varepsilon >0$, there exists a constant $\delta=\delta(F,\varepsilon)>0$ such that:

\vspace{0.25mm}
{\rm (i)}
If $r\geq 3$ and $\rho(G)\geq\sqrt{(1-\frac{1}{r}-\delta)2m}$,
then there exists a Tur\'{a}n graph $T_{n,r}$ such that
$V(T_{n,r})\subseteq V(G)$ and $d(G,T_{n,r})\leq \varepsilon m$;

\vspace{0.25mm}
{\rm (ii)} If $r=2$ and $\rho(G)\geq \sqrt{(1-\delta)m}$, or if $r=1$, $\beta'(F)\geq 2$,
and $\rho(G)\geq \sqrt{(1-\delta)m}$,
then there exist two disjoint subsets $U,V\subseteq V(G)$
with $d(G,K_{U,V})\leq \varepsilon m$.
\end{lem}

The \emph{spectral norm} $\|B\|_2$ of a matrix $B$, induced by the $\ell_2$ vector norm, is given by
\[
\|B\|_2 = \sup_{\mathbf{y} \neq \mathbf{0}} \frac{\|B\mathbf{y}\|_2}{\|\mathbf{y}\|_2}.
\]
We will use the following estimates for the spectral norms of adjacency and biadjacency matrices in terms of the number of edges.
These bounds are well known,
but we include a brief proof for completeness and ease of reference.

\begin{lem}\label{lem2.8}
Let $G$ be a graph with $m$ edges. Then, we have the following properties:

\vspace{0.25mm}
{\rm (i)}  $\|A(G)\|_2=\rho(G)\leq \sqrt{2m}$;

\vspace{0.25mm}
{\rm (ii)} If $G$ is a bipartite graph with
$A(G) = \begin{pmatrix}
\mathbf{0} & B \\
B^\top & \mathbf{0}
\end{pmatrix},
$
then $\|B\|_2=\rho(G)\leq \sqrt{m}$.
\end{lem}

\begin{proof}
For every real matrix $M$, the matrix $M^\top M$ is clearly positive semidefinite.
Therefore, we have
$\|M\|_2=\sqrt{\rho(M^\top M)}\leq \sqrt{\operatorname{tr}(M^\top M)}.$

(i) Set $M:=A(G)$. Then, $M$ is a real symmetric matrix, and thus $\|M\|_2=\rho(M)=\rho(G)$.
Moreover, since $G$ has $m$ edges,
$M$ contains exactly $2m$ entries equal to $1$ and all remaining entries are 0.
Consequently, $\operatorname{tr}(M^\top M)=2m$.
It follows that $\|M\|_2\leq\sqrt{2m}.$

(ii)
Since $B$ is the biadjacency matrix of $G$, it contains exactly $m$ entries equal to $1$. Thus,
$\|B\|_2=\sqrt{\rho(B^\top B)}\leq \sqrt{\operatorname{tr}(B^\top B)}=\sqrt{m}.$
Moreover, we know that
$$A^2(G)=\begin{pmatrix}
BB^\top & \mathbf{0}\\
\mathbf{0} & B^\top B
\end{pmatrix},$$
which implies that
$$\rho^2(G)=\rho\big(A^2(G)\big)
=\max\big\{\rho(BB^\top),\rho(B^\top B)\big\}=\rho(B^\top B).
$$
Therefore, we obtain $\|B\|_2=\rho(G)\leq \sqrt{m}.$
\end{proof}

For any two disjoint vertex subsets $U, V \subseteq V(G)$,
we write $G[U]$ for the subgraph of $G$ induced by $U$ and $G-U$ for the subgraph of $G$ induced by $V(G)\setminus U$.
Let $E_G(U)$ denote the set of edges within $U$, and let
$E_G(U,V)$ denote the set of edges with one endpoint in $U$ and the other in $V$.
The following lemma estimates the spectral radius of a graph $G$ when it is
structurally close to a complete bipartite graph $K_{U,V}$ with one part of a fixed size.

\begin{lem}\label{lem2.7G}
Let $G$ be a graph of sufficiently large size $m$ with no isolated vertices.
Fix a positive integer $a$ and a positive constant $\varepsilon \leq \frac{1}{100a}$.
Suppose there exist two disjoint subsets $U,V \subseteq V(G)$ such that
$|U| = a$ and $d(G, K_{U,V}) \le \varepsilon m$. Then

\vspace{0.25mm}
{\rm (i)}
$\rho(G)\leq \sqrt{m}+\frac{|E_G(U)|}{|U|}+O_a\big(\frac{1}{\sqrt{m}}\big)$;

\vspace{0.25mm}
{\rm (ii)} $\rho(G)\leq g_a(m)$, with equality if and only if
$G\cong K_a\vee E_{\frac ma-\frac{a-1}{2}}.$
\end{lem}

\begin{proof}
Let $U^c:=V(G)\setminus U$ and $\rho:=\rho(G)$. Then $V\subseteq U^c$.
If $\rho<\sqrt m$, both conclusions (i) and (ii) hold trivially.
Hence, we may assume
$\rho\geq \sqrt m$ for the remainder of the discussion.

Let $\mathbf{x}$ be a nonnegative unit eigenvector of $A(G)$ corresponding to $\rho(G)$.
Under the vertex partition
$V(G) = U \cup U^c$, we write $A(G)$ in the block form
\[
A(G) = \begin{pmatrix}
A\big(G[U]\big) & B \\
B^\top & A\big(G[U^c]\big)
\end{pmatrix},
\]
where $B$ is the submatrix of $A(G)$ whose rows and columns are indexed by $U$ and $U^c$, respectively.
Let $D:=A(G[U^c])$.
Put $b:=|E_G(U,U^c)|$, $d:=|E_G(U^c)|$, and $e=|E_G(U)|.$
Then $b+d+e=m$.
Since $E_G(U^c)\cap E(K_{U,V})=\varnothing$,
the edit-distance assumption gives $d\leq d(G, K_{U,V})\leq \varepsilon m$.

We first estimate the norms of the partitioned eigenvector.
Let $\alpha = \|\mathbf{x}_{U}\|_2$ and $\beta = \|\mathbf{x}_{U^c}\|_2$.
The eigenvalue equation restricted to $U^c$ yields $\rho\,\mathbf{x}_{U^c} = B^\top\mathbf{x}_{U} + D\mathbf{x}_{U^c}.$
Applying (i) of Lemma \ref{lem2.8} to the induced subgraph $G[U^c]$, we obtain
\begin{align}\label{align-G01}
\rho(D) \le \sqrt{2d} \le \sqrt{2\varepsilon m} < \sqrt{m} \le \rho,
\end{align}
which immediately implies $\alpha > 0$. Indeed, if $\alpha = 0$,
then $\mathbf{x}_{U}=\mathbf{0}$, which would force $\rho$ to be an eigenvalue of $D$.
This contradicts the strict inequality $\rho(D) < \rho$ established in \eqref{align-G01}.
 Since $\rho I - D$ is symmetric and positive definite,
its minimum eigenvalue is precisely $\rho-\rho(D)$.
By the Rayleigh quotient inequality and  the Cauchy--Schwarz inequality,
we obtain $\big(\rho - \rho(D)\big)\|\mathbf{x}_{U^c} \|_2
\le \| (\rho I - D)\mathbf{x}_{U^c} \|_2$.
Combining this with $(\rho I - D)\mathbf{x}_{U^c} = B^\top\mathbf{x}_{U}$  yields
\[
\big(\rho - \rho(D)\big)\beta=\big(\rho - \rho(D)\big)\|\mathbf{x}_{U^c} \|_2
\le \| (\rho I - D)\mathbf{x}_{U^c} \|_2
= \|B^\top\mathbf{x}_{U}\|_2.
\]
By the definition of $\|B\|_2$,
it follows that $\|B\mathbf{y}\|_2 \le \|B\|_2 \|\mathbf{y}\|_2$ for any vector $\mathbf{y}$.
Thus,
\begin{align*}
\|B^\top\mathbf{x}_{U}\|_2 \le \|B^\top\|_2 \|\mathbf{x}_{U}\|_2 = \|B\|_2 \alpha.
\end{align*}
Applying (ii) of Lemma \ref{lem2.8} to the bipartite subgraph induced by $E_G(U,U^c)$,
we obtain $\|B^\top\|_2\leq\sqrt{b}$.
Hence,
$(\rho-\rho(D))\beta\le \alpha\sqrt b\le \alpha\sqrt m.$
Furthermore, from \eqref{align-G01} we know that $\rho-\rho(D)\geq(1-\sqrt{2\varepsilon})\sqrt{m}$.
Therefore, we conclude that
$\beta\leq\alpha/(1-\sqrt{2\varepsilon})\leq 2\alpha.$

We now bound the entries of $\mathbf{x}_{U^c}$.
For every vertex $v\in U^c$, the eigenvalue equation and the
Cauchy-Schwarz inequality imply that
\begin{align*}
\rho x_v
&=\sum_{u\in N_U(v)}\!\!x_u
+\!\!\sum_{u\in N_{U^c}(v)}\!\!x_u
\leq \sqrt {|U|}\,\alpha +\sqrt{d_{U^c}(v)}\,\beta
\leq \sqrt a\,\alpha+\sqrt{d}\,\beta.
\end{align*}
From the bounds $d\leq \varepsilon m$, $\beta\leq 2\alpha$, and $\rho\geq\sqrt m$ obtained earlier,
we get
$x_v\leq\left(
\sqrt{\frac{a}{m}}+2\sqrt{\varepsilon}
\right)\alpha.$
Since $\varepsilon\leq  \frac{1}{100a}$ and $m$ is sufficiently large, we further have
$x_v\leq \frac{\alpha}{4\sqrt a}$ for every  $v\in U^c$.

To bound the bilinear term $\mathbf{x}_{U}^{\top}BD\mathbf{x}_{U^c}$,
we define
$s_v=\sum_{u\in N_U(v)}x_u$ and $t_v=\sum_{u\in N_{U^c}(v)}x_u$ for each $v\in U^c.$
Then $\mathbf{x}_{U}^{\top}B=(s_v)^\top_{v\in U^c}$ and $D\mathbf{x}_{U^c}=(t_v)_{v\in U^c}$.
As shown above, $s_v\leq \sqrt a\,\alpha$ for every $v\in U^c.$
Together with $d=|E_G(U^c)|$, this gives
\begin{align}
\mathbf{x}_{U}^{\top}BD\mathbf{x}_{U^c}=\sum_{v\in U^c}s_vt_v
=\!\!\!\sum_{v_1v_2\in E_G(U^c)}\!\!\!(s_{v_1}x_{v_2}+s_{v_2}x_{v_1})
\leq 2d\cdot\sqrt a\,\alpha\cdot\frac{\alpha}{4\sqrt a}
=\frac{1}{2}d\alpha^2.
\label{eq:BD-bound}
\end{align}

Let $\mathbf{z}:=\frac{1}{\alpha}\mathbf{x}_{U}$ and $\mathbf{u}:=\frac{1}{\sqrt a}\mathbf{1}_a$.
Then $\|\mathbf{z}\|_2=\|\mathbf{u}\|_2=1$.
By the eigenvalue equations on $U$ and $U^c$, we have
$\rho\,\mathbf{x}_{U}
=A(G[U])\mathbf{x}_{U}+B\mathbf{x}_{U^c}$ and $\rho\,\mathbf{x}_{U^c}=B^\top\mathbf{x}_{U}+D\mathbf{x}_{U^c}.$
Multiplying the first equation by $\rho$ and combining the second equation, we deduce that
\[
\rho^2\mathbf{x}_{U}=\rho A(G[U])\mathbf{x}_{U}+B\rho\mathbf{x}_{U^c}
=
\rho A(G[U])\mathbf{x}_{U}
+
BB^\top\mathbf{x}_{U}
+
BD\mathbf{x}_{U^c}.
\]
Taking the inner product with $\frac1{\alpha^2}\mathbf{x}^\top_{U}$ on both sides
and invoking \eqref{eq:BD-bound} yields
\begin{equation}\label{eq:rho-square}
\rho^2=\rho^2\|\mathbf{z}\|^2_2
\leq
\rho\,\mathbf{z}^{\top}A(G[U])\mathbf{z}
+\mathbf{z}^{\top}BB^\top\mathbf{z}+\frac 12d.
\end{equation}

{\rm (i)}
Recall that $V\subseteq U^c$.
Define $V':=\{v\in V: N_G(v)\supseteq U\}$.
Then $V'\subseteq U^c$ and $|V'|\geq |V|-|E(K_{U,V})\setminus E(G)|.$
Moreover, we have $|E(K_{U,V})\setminus E(G)|\leq d(G, K_{U,V})\leq \varepsilon m$ and
$$m-a|V|=|E(G)|-|E(K_{U,V})|
\leq d(G,K_{U,V})
\leq \varepsilon m.$$
It follows that $|V'|\geq |V|-\varepsilon m$ and $a|V|\geq (1-\varepsilon)m$. Since $\varepsilon \leq \frac{1}{100a}$,
we immediately obtain
\begin{equation}\label{eq:U20-large}
a|V'|\geq a|V|-a\varepsilon m
\geq(1-\varepsilon-a\varepsilon)m
\geq\frac 12m.
\end{equation}

Observe that $\mathbf{z}^{\top}BB^\top\mathbf{z}
=\sum_{v\in U^c}(\sum_{u\in N_U(v)}z_u)^2$.
Because $\mathbf{z}$ is a unit vector, the Cauchy-Schwarz inequality implies that
$(\sum_{u\in N_U(v)}z_u)^2\leq d_U(v)$ for any $v\in U^c$.
In particular, for every $v\in V'$, we know that $d_U(v)=|U|=a$ and thus
$$\sum_{u\in N_U(v)}z_u=\sum_{u\in U}z_u=\sqrt{a}\cdot\mathbf{z}^{\top}\mathbf{u}.$$
Recall that $b=\big|E_G(U,U^c)\big|=\sum_{v\in U^c}d_U(v).$
Consequently,
$$b-\mathbf{z}^{\top}BB^\top\mathbf{z}
=\sum_{v\in U^c}\Big(d_U(v)-\big(\sum_{u\in N_U(v)}\!\!z_u\big)^2\Big)
\geq\sum_{v\in V'}\Big(a-a\big(\mathbf{z}^{\top}\mathbf{u}\big)^2\Big).$$
Since $\|\mathbf{z}\|_2=\|\mathbf{u}\|_2=1$, we have $|\mathbf{z}^{\top}\mathbf{u}|\leq1$.
Define $\delta:=\sqrt{1-(\mathbf{z}^{\top}\mathbf{u})^2}$.
Then,
$b-\mathbf{z}^{\top}BB^\top\mathbf{z}\ge\sum_{v\in V'}a\delta^2=|V'|a\delta^2.$
Combining this with \eqref{eq:U20-large} yields
\begin{equation}\label{eq:BB-defect}
\mathbf{z}^{\top}BB^\top\mathbf{z}\leq b-|V'|a\delta^2\le b-\frac{1}{2}\delta^2m.
\end{equation}

Since both $\mathbf{z}$ and $\mathbf{u}$ are nonnegative unit vectors,
we have $\|\mathbf{z}+\mathbf{u}\|_2\leq2$ and $0\leq\mathbf{z}^{\top}\mathbf{u}\leq1$.
Consequently, we deduce that
\[
\|\mathbf{z}-\mathbf{u}\|_2^2
=(\mathbf{z}-\mathbf{u})^\top(\mathbf{z}-\mathbf{u})=
2(1-\mathbf{z}^{\top}\mathbf{u})
\leq
2\bigl(1-(\mathbf{z}^{\top}\mathbf{u})^2\bigr)
=
2\delta^2.
\]
By applying the Cauchy-Schwarz inequality and the definition of the matrix operator norm
(specifically, $\|B\mathbf{y}\|_2 \le \|B\|_2 \|\mathbf{y}\|_2$ for any vector $\mathbf{y}$), we obtain:
\begin{align*}
\left| \mathbf{z}^{\top}A(G[U])\mathbf{z} - \mathbf{u}^{\top}A(G[U])\mathbf{u} \right|
&= \left| (\mathbf{z} + \mathbf{u})^{\top} A(G[U]) (\mathbf{z} - \mathbf{u}) \right|
\leq \|\mathbf{z} + \mathbf{u}\|_2  \|A(G[U])(\mathbf{z} - \mathbf{u})\|_2 \\
&\leq \|\mathbf{z} + \mathbf{u}\|_2\|A(G[U])\|_2 \|\mathbf{z} - \mathbf{u}\|_2
\leq 2\sqrt{2}\delta\|A(G[U])\|_2.
\end{align*}
Note that
$\mathbf{u}^{\top}A(G[U])\mathbf{u}
=\frac{1}{a}\mathbf{1}_a^{\top}A(G[U])\mathbf{1}_a
=\frac{2e}{a}.$
Then,
$\mathbf{z}^{\top}A(G[U])\mathbf{z}
\leq \frac{2e}{a}+2\sqrt{2}\delta\|A(G[U])\|_2.$
Since $\|A(G[U])\|_2=\rho(G[U])\leq \rho(K_{|U|})=a-1$, it follows that
\begin{align}\label{eq:U1-internal}
\mathbf{z}^{\top}A(G[U])\mathbf{z}
\leq \frac{2e}{a}+2\sqrt{2}(a-1)\delta.
\end{align}
Substituting \eqref{eq:BB-defect} and \eqref{eq:U1-internal} into
\eqref{eq:rho-square}, and recalling $b+d+e=m$, we obtain
\begin{align*}
\rho^2\leq
\rho\big(\frac{2e}{a}+2\sqrt{2}(a-1)\delta\big)+b-\frac 12\delta^2m+\frac12d
\leq\frac{2e}{a}\rho+(m-e)+2\sqrt{2}(a-1)\rho\delta-\frac 12\delta^2m.
\end{align*}
From Lemma \ref{lem2.8}, we know that $\rho\leq\sqrt{2m}.$
Thus, we can deduce that
\begin{align*}
2\sqrt{2}(a-1)\rho\delta-\frac m2\delta^2
\leq 4(a-1)\sqrt{m}\,\delta-\frac m2\delta^2
=-\frac m2\Big(\delta-\frac{4(a-1)}{\sqrt{m}}\Big)^2+8(a-1)^2.
\end{align*}
Consequently,
$\rho^2\leq \frac{2e}{a}\rho+(m-e)+8(a-1)^2.$
Equivalently,
$\left(\rho-\frac ea\right)^2
\leq m-e+\frac{e^2}{a^2}+8(a-1)^2=m+O_a(1)$,
because $|U|=a$ and $e=|E_G(U)|\leq \binom a2$.
Therefore,
\[
\rho
\leq
\frac ea+\sqrt{m+O_a(1)}
=
\sqrt m+\frac ea+O_a\Big(\frac1{\sqrt m}\Big).
\]
This establishes conclusion (i).

\medskip

{\rm (ii)}
Recall that
$g_a(m)=\frac12\big(a-1+\sqrt{4m-a^2+1}\big).$
Clearly, $g_a(m)$ is the positive root of the polynomial
$x^2-(a-1)x-\big(m-\binom{a}{2}\big)$,
and $g_a(m)=\sqrt m+\frac{a-1}{2}+O_a\big(\frac1{\sqrt m}\big)$.
By conclusion (i),
$\rho\leq\sqrt m+\frac ea+O_a\big(\frac1{\sqrt m}\big).$
If $e\leq \binom a2-1$, then $\frac{e}{a}\leq\frac{a-1}{2}-\frac1a,$ and consequently
\begin{align*}
\rho\leq \sqrt m+\frac{e}{a} +O_a\Big(\frac1{\sqrt m}\Big)
<g_a(m),
\end{align*}
as desired.

It remains the case when $e=\binom a2$.
Now, we have $G[U]\cong K_a$.
Thus, the eigenvalue equations on $U$ and $U^c$ become
$\rho\mathbf{x}_{U}
=A(K_a)\mathbf{x}_{U}+B\mathbf{x}_{U^c}$
and
$\rho\mathbf{x}_{U^c}
=B^{\top}\mathbf{x}_{U}+D\mathbf{x}_{U^c}$.
Multiplying the first equation by $\rho$ and combining the second equation,
we deduce that
\[
\rho^2\mathbf{x}_{U}=\rho A(K_a)\mathbf{x}_{U}+B\rho\mathbf{x}_{U^c}
=\rho A(K_a)\mathbf{x}_{U}+BB^{\top}\mathbf{x}_{U}
+BD\mathbf{x}_{U^c}.
\]
Taking the inner product with $\mathbf{x}^\top_{U}$ on both sides and recalling $\|\mathbf{x}_{U}\|_2=\alpha$,
we obtain
\[
\begin{aligned}
\rho^2\alpha^2
&=\rho\mathbf{x}_{U}^{\top}A(K_a)\mathbf{x}_{U}
+\|B^{\top}\mathbf{x}_{U}\|_2^2
+\mathbf{x}_U^{\top}BD\mathbf{x}_{U^c}.
\end{aligned}
\]

Note that
$\mathbf{x}_{U}^{\top}A(K_a)\mathbf{x}_{U}\leq\rho(K_a)\|\mathbf{x}_{U}\|^2_2=(a-1)\alpha^2$.
Recall that $\|B^{\top}\|_2\leq \sqrt{b}$. Thus, we have
$\|B^{\top}\mathbf{x}_{U}\|_2^2\leq\|B^{\top}\|_2^2\|\mathbf{x}_{U}\|_2^2\leq b\alpha^2$.
Together with \eqref{eq:BD-bound}, this yields
$\rho^2\alpha^2
\leq\big((a-1)\rho+b+\frac 12d\big)\alpha^2.$
Hence, $\rho^2\leq(a-1)\rho+b+\frac 12d$.
Using $b+d=m-\binom a2,$
we further obtain
\begin{align*}
\rho^2\leq(a-1)\rho+m-\binom a2-\frac 12d
\leq(a-1)\rho+m-\binom a2.
\end{align*}
Therefore,
$\rho^2-(a-1)\rho-\big(m-\binom a2\big)\leq 0.$
Since $g_a(m)$ is the positive root of the corresponding quadratic
equation, it follows that $\rho(G)\leq g_a(m).$

We next characterize the equality case.
If $\rho(G)=g_a(m),$
then equality must hold throughout the preceding chain of inequalities.
In particular, we have $d=0$, thus $U^c$ is an independent set.
Moreover, equality also holds in
$\mathbf{x}_{U}^{\top}A(K_a)\mathbf{x}_{U}\leq\rho(K_a)\|\mathbf{x}_{U}\|^2_2,$
which implies that $\mathbf{x}_{U}$ is an eigenvector corresponding to $\rho(K_a)$.
Since $\|\mathbf{x}_{U}\|_2=\alpha$ and the eigenspace of $A(K_a)$ corresponding to $\rho(K_a)$
is generated by $\mathbf{1}_a$, it follows that
$\mathbf{x}_U=\frac{\alpha}{\sqrt a}\mathbf{1}_a.$
Furthermore, equality in
$\|B^\top\mathbf{x}_{U}\|_2^2\leq b\alpha^2$
now gives $\frac{1}{a}\alpha^2\sum_{v\in U^c}d^2_U(v)=\alpha^2\sum_{v\in U^c}d_U(v),$
which simplifies to $\sum_{v\in U^c} d_U(v)\big(a-d_U(v)\big)=0.$

Since $0\leq d_U(v)\leq a$ for every $v\in U^c$, each summand is
nonnegative. Therefore,
$d_U(v)\in\{0,a\}$ for every $v\in U^c$.
On the other hand, since $U^c$ is an
independent set and $G$ has no isolated vertices, we cannot have
$d_U(v)=0$ for any $v\in U^c$. Consequently,
$d_U(v)=a$ for every $v\in U^c$,
which further implies that $G\cong K_a\vee E_{|U^c|}$,
where $|U^c|=\big(m-\binom a2\big)/a=\frac ma-\frac{a-1}{2}.$
This establishes conclusion (ii).
\end{proof}

\section{Sharpness construction for Theorem~\ref{thm1.1}}\label{sec3}

%

By specializing a result of Sauer \cite{Sauer1970} on hypergraphs
to the case of ordinary graphs, we obtain the following existence lemma.

\begin{lem}\label{lem2.5T}\emph{(\cite{Sauer1970})}
For any integers $\ell \ge 3$ and $g \ge 1$, and any even $n$ divisible by $g$ satisfying
$$n\ge 2 \Big( \sum_{i=1}^{\ell-2} (g-1)^i + (g-1)^{\ell-2} \Big) + 1,$$
there exists a $g$-regular graph $H$ on $n$ vertices with girth at least $\ell$.
\end{lem}

Recall that $m$ is sufficiently large.
Let $s\geq 3$ and $t\geq 2$ be constant integers,
and let $\phi$ be the unique positive integer divisible by $4s$ that satisfies
\begin{equation}\label{align-00G}
2s(\phi-4s)+(\phi-4s)^2< m \le 2s\phi + \phi^2.
\end{equation}
By setting $g=4s$ and $\ell=5$ in Lemma \ref{lem2.5T},
we obtain a $4s$-regular $C_4$-free graph $H^*$ of order $\phi$.
Define $\psi:=\lceil \frac{m-2s\phi}{\phi}\rceil$.
Then $\psi-1<\frac{m-2s\phi}{\phi}\leq \psi$,
and hence
$\phi(\psi-1)+2s\phi<m\leq \phi\psi+2s\phi.$
Note that $e(H^* \vee E_{\psi-1})=\phi(\psi-1)+2s\phi$
and $e(H^* \vee E_{\psi})=\phi\psi+2s\phi$.
Consequently, $e(H^* \vee E_{\psi-1})<m\leq e(H^* \vee E_{\psi})$.

Let $\mathbb{G}(m,H^*)$ denote the family of all $m$-edge graphs $G$ satisfying
$H^* \vee E_{\psi-1} \subseteq G \subseteq H^* \vee E_{\psi}$.
The graph $H^*\vee E_{\psi}$ is obtained from $H^*\vee E_{\psi-1}$ by adding one new vertex together with its $\phi$ incident edges.
Since $e(H^* \vee E_{\psi-1})<m\leq e(H^* \vee E_{\psi})$,
it is clear that $\mathbb{G}(m,H^*)\neq\varnothing$.

As a preparatory step,
we now show that every graph in $\mathbb{G}(m,H^*)$ achieves the number of copies of $K_{s,t}^+$ specified in Theorem~\ref{thm1.1},
and provide the related proofs.
This auxiliary tool will facilitate the verification of our main theorems.

\begin{lem}\label{lem2.5G}
For all sufficiently large $m$ and every graph $G\in\mathbb{G}(m,H^*)$, we have

\vspace{0.25mm}
\textnormal{(i)} $\rho(G)> g_{s-1}(m)$;

\vspace{0.25mm}
\textnormal{(ii)} $N(K_{s,t}^+, G) = \Theta\big(m^{\frac{s+t-1}{2}}\big)$.
\end{lem}

\begin{proof}
(i) Denote $G'=H^* \vee E_{\psi-1}$ and $m'=e(G')$.
We first prove that
\begin{equation}\label{align-03G}
\rho(G') = \sqrt{m'} + s + O\Big(1\big/\sqrt{m'}\Big).
\end{equation}
Partition $V(G')$ as $\Pi: V(G')=V(H^*)\cup V(E_{\psi-1})$.
Note that $H^*$ is $4s$-regular and $|H^*|=\phi$. Then, $m'=\phi(\psi-1)+2s\phi$.
Moreover, every vertex in $V(H^*)$ is adjacent to all $\psi-1$ vertices in $V(E_{\psi-1})$,
and every vertex in $V(E_{\psi-1})$ is adjacent to all $\phi$ vertices in $V(H^*)$.
Hence, the partition $\Pi$ is equitable, and thus
$\rho(G')$ is the largest eigenvalue of the quotient matrix
\[
B_{\Pi} = \begin{bmatrix}
4s & \psi-1 \\
\phi & 0
\end{bmatrix}.
\]
Consequently, $\rho(G')$ is the largest root of the characteristic polynomial:
\[
f(x) := \det(xI_2-B_{\Pi}) = x^2 - 4sx - \phi(\psi-1).
\]
Then $\rho(G') = 2s + \sqrt{\phi(\psi-1)+4s^2}$.
Rearranging (\ref{align-00G}) gives $\phi(\phi-8s)+8s^2<m-2s\phi\leq \phi^2$.
Based on the definition of $\psi$, we obtain $\phi-8s<\psi \le \phi.$
Set $\mu:=\phi-\psi$. Then $0\leq\mu<8s$,
where $s\ge3$ is a constant integer.
Thus, we can deduce that
\begin{equation}\label{align-05G}
\rho(G')= 2s + \sqrt{\phi(\phi\!-\!\mu\!-\!1)+4s^2} = 2s + \phi - \frac{\mu\!+\!1}{2} + O\big(1\big/\phi\big).
\end{equation}
Since $m' = \phi(\psi-1)+ 2s\phi = \phi^2-(\mu+1-2s)\phi$,
solving this quadratic equation for $\phi$ yields
\begin{equation}\label{align-04G}
\phi =\big(\frac{\mu\!+\!1}{2}\!-\!s\big) + \sqrt{m'\!+\!\big(\frac{\mu\!+\!1}{2}\!-\!s\big)^2}
= \sqrt{m'}+ \frac{\mu\!+\!1}{2} - s + O\big(1\big/\sqrt{m'}\big).
\end{equation}
This implies that $\sqrt{m'}-s\leq\phi\leq\sqrt{m'}+4s$, and thus $\phi=\sqrt{m'}+O(1)$.
Substituting \eqref{align-04G} into \eqref{align-05G},
the term $\frac{\mu+1}2$ cancels out,
and equality \eqref{align-03G} follows immediately.

Now, we establish statement (i).
By the definition of \(g_{s-1}(m)\), since \(s\ge 3\), we have
\[
g_{s-1}(m)
=\frac{s-2+\sqrt{4m-s^2+2s}}{2}
<\sqrt m+\frac{s-2}{2}.
\]
Since \(G'=H^*\vee E_{\psi-1} \subseteq G \subseteq H^* \vee E_{\psi}\),
we have \(\rho(G)\geq \rho(G')\) and \(m'\le m\le m'+\phi\).
Moreover, for sufficiently large \(m\),
$\phi\le \sqrt{m'}+4s\le \frac{s}{2}\sqrt{m'}\le \frac{s}{2}\sqrt m.$
Hence
$m-\phi\ge m-\frac{s}{2}\sqrt m
\ge \left(\sqrt m-\frac{s}{2}\right)^2,$
and therefore
$\sqrt{m-\phi}+s-1
\ge \sqrt m+\frac{s-2}{2}.$
Combining this with \eqref{align-03G}, we obtain
\[
\rho(G)\ge \rho(G')\ge \sqrt{m'}+s-1
\ge \sqrt{m-\phi}+s-1
\ge \sqrt m+\frac{s-2}{2}
>g_{s-1}(m).
\]

\medskip
(ii) Denote $G''=H^*\vee E_{\psi}$.
We first demonstrate that
\begin{equation}\label{align-000G}
N(K_{s,t}^+,G'')=\Theta(m^{\frac{s+t-1}{2}}).
\end{equation}
Recall that $m'\le m \le m'+\phi$, $|H^*|=\phi=\sqrt{m'}+O(1)$ and $\phi-8s<\psi \le \phi,$
where $s$ is constant and $m$ is sufficiently large.
Then $\phi,\psi = \Theta(\sqrt{m})$.
Consequently, both $G'$ and $G''$ contain copies of $K_{s,t}^+$,
and proving \eqref{align-000G}
is equivalent to showing that $N(K_{s,t}^+, G'') = \Theta(\phi^{s+t-1})$.

Let \(\mathcal{F}\) be the family of all copies of \(K_{s,t}^{+}\) contained in \(G''\).
Given an $F\in\mathcal{F}$, we may decompose $V(F)$ as $X\cup Y$, where $F[X]\cong K_2\cup E_{s-2}$ and $F[Y]\cong E_t$.
We further partition $X$ and $Y$ by their intersections with the pre-fixed vertex subsets $V(H^*)$ and $ V(E_\psi)$ of $G''$:
\[
X_H = X \cap V(H^*), X_E = X \cap V(E_\psi), Y_H = Y \cap V(H^*), Y_E = Y \cap V(E_\psi).
\]
Let
$$
\mathcal{F}_1:=\{F\in\mathcal{F}:Y_E\neq\varnothing\}
\quad\text{and}\quad
\mathcal{F}_2:=\{F\in\mathcal{F}:Y_E=\varnothing\}.
$$
Then
$\mathcal{F}=\mathcal{F}_1\mathbin{\dot\cup}\mathcal{F}_2$.
For all $F\in\mathcal{F}$, we divide the counting analysis into two possible cases.

\medskip
\noindent\textbf{Case 1: $Y_E\neq\varnothing$.}

Let $F\in\mathcal{F}_1$. Choose a vertex $y\in Y_E$. Since
$y\in Y$, every vertex of $X$ is adjacent to $y$ in $F$, and hence
$X\subseteq N_F(y)\subseteq N_{G''}(y).$
On the other hand, since $y\in V(E_\psi)$ and
$G''=H^*\vee E_\psi$, we have
$N_{G''}(y)=V(H^*).$
Therefore,
$X\subseteq V(H^*).$
Thus, in Case~1, every admissible choice of $X$ is contained entirely
in $V(H^*)$.

We first derive an upper bound for $|\mathcal{F}_1|$.
For every $F\in\mathcal{F}_1$, let $uv$ be the added edge of the
corresponding copy of $K_{s,t}^+$. Since $X\subseteq V(H^*)$, we have
$uv\in E(H^*)$. Recall that $H^*$ is $4s$-regular on $\phi$ vertices,
and hence
$e(H^*)=2s\phi.$
Thus, there are at most $2s\phi$ choices for $uv$. Once $uv$ is fixed,
the remaining $s-2$ vertices of $X$ can be chosen in at most
$\binom{\phi-2}{s-2}$ ways.
After $X$ is fixed, the set $Y$ can be chosen in at most
$\binom{\phi+\psi-s}{t}$ ways. Consequently,
$$
\begin{aligned}
|\mathcal{F}_1|
\leq
2s\phi
\binom{\phi-2}{s-2}
\binom{\phi+\psi-s}{t}
=O(\phi^{s-1})\cdot O(\phi^t)
=O(\phi^{s+t-1}),
\end{aligned}
$$
where we have used $\psi=\Theta(\phi)$ and the fact that $s$ and $t$
are fixed.

We next establish the corresponding lower bound.
Choose an edge $uv\in E(H^*)$, an $(s-2)$-set
$A'\subseteq V(H^*)\setminus\{u,v\}$, and a $t$-set
$B\subseteq V(E_\psi)$. Put
$A:=A'\cup\{u,v\}.$
Since $G''=H^*\vee E_\psi$, every edge between $A$ and $B$ is present
in $G''$. Together with the edge $uv\in E(H^*)$, these edges contain
a copy of $K_{s,t}^+$ whose part of size $s$ is $A$, whose part of
size $t$ is $B$, and whose added edge is $uv$. Since
$B\subseteq V(E_\psi)$, every such copy belongs to $\mathcal{F}_1$.
The number of triples $(uv,A',B)$ is
$$
2s\phi
\binom{\phi-2}{s-2}
\binom{\psi}{t}
=\Theta(\phi^{s+t-1}),
$$
because $\psi=\Theta(\phi)$.
A fixed copy of $K_{s,t}^+$ can arise from at most a constant number
of such triples, where the constant depends only on $s$ and $t$.
Indeed, all vertices involved in such a triple belong to the fixed
vertex set of the copy, so the number of possible choices of the
distinguished edge $uv$ and the corresponding sets $A'$ and $B$ is
bounded by a constant depending only on $s$ and $t$. Therefore,
$|\mathcal{F}_1|=\Omega(\phi^{s+t-1}).$
Combining the upper and lower bounds, we conclude that
$|\mathcal{F}_1|=\Theta(\phi^{s+t-1}).$

\medskip
\noindent\textbf{Case 2: $Y_E = \varnothing$.}

In this case, $Y \subseteq V(H^*)$ and $H^*$ contains a complete bipartite subgraph $K_{X_H,Y}$.
Moreover, the number of ways to choose $Y$ from $V(H^*)$ is at most $\binom{\phi}{t} =\Theta(\phi^t)$.

Since $F[X]$ contains exactly one edge while $X_E \subseteq V(E_\psi)$ is an independent set,
we must have $|X_H|\geq1$.
Furthermore, the assumptions that $|Y|=t\geq2$ and that $H^*$ is $C_4$-free force $|X_H|=1$.
It follows that $|X_E|=s-1.$

Let $X_H = \{x\}$ for some $x \in V(H^*)$.
Since $Y\subseteq V(H^*)$ and $x$ is adjacent to all vertices of $Y$,
we must have $Y \subseteq N_{H^*}(x)$.
Recall that $H^*$ is $4s$-regular.
Then $d_{H^*}(x)=4s$. Consequently, once $x$ is chosen,
there are only $\binom{4s}{t}$ choices for $Y$.
Thus, we can choose $x \in V(H^*)$ in $\phi$ ways, $Y$ in $\binom{4s}{t}$ ways, and
$X_E \subseteq V(E_\psi)$ in $\binom{\psi}{s-1}$ ways.
Note that $t\geq2$. Therefore,
\[
|\mathcal{F}_2|
=O\left(\phi\binom{4s}{t}\binom{\psi}{s-1}\right)
=o(\phi^{s+t-1}).
\]

Since
$\mathcal{F}=\mathcal{F}_1\mathbin{\dot\cup}\mathcal{F}_2$,
we have
\[
N(K_{s,t}^{+},G'')
=
|\mathcal{F}|
=
|\mathcal{F}_1|+|\mathcal{F}_2|
=
\Theta(\phi^{s+t-1}).
\]

We have proved that
$N(K_{s,t}^{+},G'')=O(\phi^{s+t-1})$.
We next establish the corresponding lower bound for $G'$.
Recall that
$G'=H^*\vee E_{\psi-1}$,
where $H^*$ is a $4s$-regular graph on $\phi$ vertices. Hence
$e(H^*)=2s\phi.$
Moreover, since
$\phi-8s<\psi\leq\phi$,
we have
$\psi-1=\Theta(\phi)$
for sufficiently large $m$.
To construct a copy of $K_{s,t}^{+}$ in $G'$, first choose an edge
$uv\in E(H^*)$. There are $2s\phi$ choices for $uv$.
Next, choose an $(s-2)$-set
$A'\subseteq V(H^*)\setminus\{u,v\}$,
and a $t$-set
$B\subseteq V(E_{\psi-1})$.
Put
$A:=A'\cup\{u,v\}$.
Since $G'=H^*\vee E_{\psi-1}$, every edge between $A$ and $B$
is present in $G'$, and $uv\in E(H^*)$. Therefore, the edges
between $A$ and $B$, together with the edge $uv$, contain a copy
of $K_{s,t}^{+}$ whose added edge is $uv$.
The number of triples $(uv,A',B)$ is
$
2s\phi
\binom{\phi-2}{s-2}
\binom{\psi-1}{t}
=
\Theta\bigl(\phi^{s+t-1}\bigr).
$
A fixed copy of $K_{s,t}^{+}$ can arise from at most a constant
number of such triples, where the constant depends only on $s$
and $t$. Hence
$N(K_{s,t}^{+},G')=\Omega(\phi^{s+t-1})$.

Now let $G\in\mathbb{G}(m,H^*)$ be arbitrary. By the definition of
$\mathbb{G}(m,H^*)$, we have
$G'\subseteq G\subseteq G''$.
Therefore,
\[
N(K_{s,t}^{+},G')
\leq
N(K_{s,t}^{+},G)
\leq
N(K_{s,t}^{+},G'').
\]
Combining the lower bound for $G'$ with the upper bound for $G''$,
we obtain
$N(K_{s,t}^{+},G)
=\Theta(\phi^{s+t-1})$.
Finally, since
$\phi=\Theta(\sqrt{m}),$
it follows that
$N(K_{s,t}^{+},G)=\Theta(
m^{\frac{s+t-1}{2}}).
$
This proves statement {\rm (ii)}.
\end{proof}

\section{Properties of an $\varepsilon$-core subgraph $H$ of the counterexample graph $G$}\label{sec4}

We first introduce a key definition that underpins the subsequent arguments.

\begin{definition}\label{def3.1G}
Let $\varepsilon\in(0,\frac15)$ be a constant, and define
$\Phi(G):=\rho(G)\big/\sqrt{e(G)}$.
A subgraph $G'$ of $G$ with no isolated vertices is called an
{\bf \emph{$\varepsilon$-dense subgraph}} of $G$ if
$$(1\!-\!\varepsilon) e(G)<e(G')<e(G)~~
\text{and}~~\Phi(G')\!-\!\Phi(G)\geq\frac{\varepsilon\big(e(G)\!-\!e(G')\big)}{2e(G)}.$$
An {\bf \emph{$\varepsilon$-core}} is a graph with no isolated vertices
and no $\varepsilon$-dense subgraphs.
\end{definition}

By Lemma \ref{lem2.5G},
the order of the leading term $m^{\frac{s+t-1}{2}}$ in Theorem \ref{thm1.1} is best possible.

In the remainder of the proof, we show that
$N(K_{s,t}^+,G)=\Omega(m^{\frac{s+t-1}{2}})$
for every sufficiently large $m$ and every $m$-edge graph $G$
with $\delta(G)\geq1$ and $\rho(G)>g_{s-1}(m)$.
For convenience, we fix the following hierarchy of constants:
\begin{align}\label{align-7G}
0<\varepsilon \ll \varepsilon_{0}\ll \varepsilon_{1}
\ll \varepsilon_{2}\ll \varepsilon_3\ll 1.
\end{align}
Here, the relation $\varepsilon \ll \eta$ means that $\varepsilon$
is chosen sufficiently small as a function of $\eta$ so that all
inequalities appearing below hold.
Suppose, for contradiction, that the desired conclusion fails.
Then there exist arbitrarily large integers $m$ and $m$-edge graphs
$G$ with $\delta(G)\geq1$ and
$\rho(G)>g_{s-1}(m)$
such that
\begin{align}\label{align-6G}
N(K_{s,t}^+,G)
<\varepsilon_1 m^{\frac{s+t-1}{2}}.
\end{align}

In the proof of Theorem \ref{thm1.1},
we seek an \(\varepsilon_0\)-core subgraph $H$ of $G$
such that $e(H)>(1-\varepsilon_0)e(G)$.
So, $H$ contains no \(\varepsilon_0\)-dense subgraphs.
We proceed by constructing a sequence of graphs $G_{(1)}\supset G_{(2)}\supset\cdots\supset G_{(\ell)}$,
where $G_{(1)}=G$ and $G_{(i+1)}$ is an $\varepsilon_{0}$-dense subgraph of $G_{(i)}$ for $1\leq i\leq \ell-1$.
The process stops as soon as either
\(e(G_{(1)})-e(G_{(\ell)})\geq2\lfloor\varepsilon_{0} m\rfloor\)
or \(G_{(\ell)}\) contains no \(\varepsilon_{0}\)-dense subgraphs.

\begin{lem}\label{Lem4.1}
We have $e(G_{(1)})-e(G_{(\ell)})<\lfloor\varepsilon_{0} m\rfloor$.
\end{lem}

\begin{proof}
Assume, for the sake of contradiction, that
$e(G_{(1)})-e(G_{(\ell)})\geq \lfloor \varepsilon_0 m\rfloor$.
For each $i\in\{1,\ldots,\ell-1\}$, since $e(G_{(i)})\leq m$
and $G_{(i+1)}$ is an $\varepsilon_0$-dense subgraph of $G_{(i)}$, it follows that
$$\Phi(G_{(i+1)})-\Phi(G_{(i)})
\geq\varepsilon_0\big(e(G_{(i)})\!-\!e(G_{(i+1)})\big)\Big/2e(G_{(i)})\geq
\frac{\varepsilon_0}{2m}\big(e(G_{(i)})-e(G_{(i+1)})\big).$$
Summing this inequality over $i=1,\ldots,\ell-1$ gives
$$\Phi(G_{(\ell)})-\Phi(G_{(1)})
\geq\frac{\varepsilon_0}{2m}
\sum_{i=1}^{\ell-1}\big(e(G_{(i)})-e(G_{(i+1)})\big)
=\frac{\varepsilon_0}{2m}
\big(e(G_{(1)})-e(G_{(\ell)})\big)
\geq\frac{\varepsilon_0}{2m}\lfloor \varepsilon_0 m\rfloor.$$
This yields $\Phi(G_{(\ell)})-\Phi(G_{(1)})>\frac14\varepsilon_0^2.$
Since $G_{(1)}=G$ and $\rho(G)>g_{s-1}(m)>\sqrt{m}$,
we know that $\Phi(G_{(1)})=\rho(G)/\sqrt m>1.$
Consequently,
\begin{equation}\label{align-08G}
\Phi(G_{(\ell)})
>1+\frac{1}{4}\varepsilon_0^2.
\end{equation}

On the other hand, the choice of $\ell$ implies
$e(G_{(1)})-e(G_{(\ell-1)})<2\lfloor\varepsilon_0 m\rfloor.$
Moreover, since $G_{(\ell)}$ is an $\varepsilon_0$-dense subgraph of $G_{(\ell-1)}$,
we have $e(G_{(\ell-1)})-e(G_{(\ell)})<\varepsilon_0\cdot e(G_{(\ell-1)})$.
Therefore,
$e(G_{(1)})-e(G_{(\ell)})<4\lfloor \varepsilon_0 m\rfloor$.
Note that $e(G_{(1)})=m$ and $\varepsilon_0$ is sufficiently small.
Thus, we obtain $\frac{m}{2}<e(G_{(\ell)})<m$.
Combining this estimate with \eqref{align-6G} yields
\[
N\big(K_{s,t}^+,G_{(\ell)}\big)
\leq N\big(K_{s,t}^+,G\big)
=o\big(m^{\frac{s+t}{2}}\big)
=o\big(\big(e(G_{(\ell)})\big)^{\frac{s+t}{2}}\big).
\]
By Lemma~\ref{thm2.2G}, we obtain
$\rho(G_{(\ell)})\leq\sqrt{\big(1+o(1)\big)e(G_{(\ell)})}.$
Thus, $\Phi(G_{(\ell)})=\rho(G_{(\ell)})\big/\sqrt{e(G_{(\ell)})}< 1 + \frac{1}{8}\varepsilon_0^2$,
which contradicts \eqref{align-08G}.
We thus conclude that
$e(G_{(1)})-e(G_{(\ell)})<\lfloor \varepsilon_0 m\rfloor$.
\end{proof}

In what follows, we focus on the structural analysis of the terminal graph $G_{(\ell)}$, rather than that of the original graph $G$.
For convenience, set $H:=G_{(\ell)}$, $h:=e(H)$ and $\rho:=\rho(H)$.
By Lemma \ref{Lem4.1} and the definition of $H$,
the graph $H$ contains no \(\varepsilon_{0}\)-dense subgraphs;
in other words, $H$ is an $\varepsilon_0$-core subgraph of $G$.
By the Perron--Frobenius theorem, there exists a nonnegative unit eigenvector
$\mathbf{x}=(x_1,\ldots,x_n)^\top$
corresponding to the spectral radius $\rho$.
Let $u^*\in V(H)$ be a vertex satisfying $x_{u^*}=\max_{v\in V(H)}x_v$.
The graph $H$ enjoys several useful structural properties that need not hold for $G$, as established in Lemmas \ref{Lem4.2}--\ref{Lem4.12C}.

\begin{lem}\label{Lem4.2}
We have $\sqrt{h}<g_{s-1}(h)<\rho<\sqrt{(1+2\varepsilon)h}$.
\end{lem}

\begin{proof}
By Lemma~\ref{Lem4.1}, we have $m-h < \varepsilon_0 m$. Since $\varepsilon_0 \ll 1$, it follows that
\begin{align}\label{ali-01}
\frac{m}{2} < h \le m.
\end{align}

It remains to verify the desired lower bound on $\rho(H)$. If $\ell=1$, then $H=G_{(1)}=G$, and the required inequality follows directly from the assumption on $G$.
Suppose now that $\ell \ge 2$.
Since each $G_{(i+1)}$ is an $\varepsilon_0$-dense subgraph of $G_{(i)}$, we have
\[
\Phi(G_{(i+1)})-\Phi(G_{(i)})
\ge
\frac{\varepsilon_0\big(e(G_{(i)})-e(G_{(i+1)})\big)}{2e(G_{(i)})}
\geq
\frac{\varepsilon_0\big(e(G_{(i)})-e(G_{(i+1)})\big)}{2m}
\]
for every $i \in [\ell-1]$. Summing this inequality over all $i \in [\ell-1]$ yields
\begin{equation}\label{eq:phi-lower}
\Phi(H) = \Phi(G_{(\ell)})
\geq  \Phi(G_{(1)}) +\sum_{i=1}^{\ell-1}\frac{\varepsilon_0\big(e(G_{(i)})-e(G_{(i+1)})\big)}{2m}
 =\Phi(G_{(1)}) + \frac{\varepsilon_0 (m-h)}{2m}.
\end{equation}

To show that $\Phi(H) >\frac{s-2+\sqrt{4h-(s-1)^2+1}}{2\sqrt{h}}$,
let us define the auxiliary function
$$f_{s-1}(x) = \frac{s-2+\sqrt{4x-(s-1)^2+1}}{2\sqrt{x}}$$
for $x > \frac{(s-1)^2-1}{4}$.
By our assumption on $G = G_{(1)}$,
we have $\Phi(G_{(1)}) > f_{s-1}(m)$.
Thus, from \eqref{eq:phi-lower}, it suffices to show that
\begin{equation}\label{eq:g-ineq}
f_{s-1}(m) + \frac{\varepsilon_0 (m-h)}{2m} > f_{s-1}(h).
\end{equation}
Since $\ell\geq 2$ and each step deletes at least one edge, we have $h < m$.
Then there exists some $\xi \in (h, m)$ such that
\begin{equation}\label{eq:mvt}
f_{s-1}(h) - f_{s-1}(m) = -f_{s-1}'(\xi)(m-h).
\end{equation}
A direct computation of the derivative of $f_{s-1}(x)$ yields
\[
f_{s-1}'(x) = -\frac{s-2}{4x^{\frac32}} +\frac{(s-1)^2-1}{8x^2\sqrt{1 - \frac{(s-1)^2-1}{4x}}}
 = -\frac{s-2}{4x^{\frac32}} + O(x^{-2}).
\]
Since $h > m/2$ and $\xi \in (h, m)$, we have $\xi > m/2$. Therefore,
$$-f_{s-1}'(\xi) = \frac{s-2}{4\xi^{\frac{3}{2}}} + O(\xi^{-2})< \frac{\varepsilon_0}{2m},$$
as $m$ is sufficiently large.
Substituting this back into \eqref{eq:mvt} gives
$$f_{s-1}(h) - f_{s-1}(m) = -f_{s-1}'(\xi)(m-h) < \frac{\varepsilon_0 (m-h)}{2m},$$
which proves \eqref{eq:g-ineq}. Consequently, we obtain
$\Phi(H) > f_{s-1}(h) = \frac{s-2+\sqrt{4h-(s-1)^2+1}}{2\sqrt{h}},$
which is equivalent to $\rho> g_{s-1}(h)$.
Furthermore, by \eqref{equ-001}, we have $g_{s-1}(h)>\sqrt{h}$.

By \eqref{align-6G} and \eqref{ali-01}, we have
\begin{align}\label{align-002T}
N(K_{s,t}^+,H)
\leq N(K_{s,t}^+,G)
< \varepsilon_1 m^{\frac{s+t-1}2}
\leq 2^{\frac{s+t-1}2}\varepsilon_1 h^{\frac{s+t-1}2}.
\end{align}
Then
$N(K_{s,t}^+,H)=o\bigl(h^{(s+t)/2}\bigr)$.
Since $\chi(K_{s,t}^+)=3$, Lemma~\ref{thm2.2G} implies that, for sufficiently large $h$,
$\rho<\sqrt{(1+2\varepsilon)h};$
otherwise, there would exist a constant $\delta>0$ such that
$N(K_{s,t}^+,H)\geq \delta h^{(s+t)/2},$
contradicting \eqref{align-002T}.
This  completes the proof.
\end{proof}

\begin{lem}\label{Lem4.3}
For any edge $uv\in E(H)$, we have
$x_{u}x_{v}> \frac{1-\varepsilon_0}{4\sqrt{h}}.$
\end{lem}

\begin{proof}
Suppose to the contrary,
then there exists an edge $uv\in E(H)$ such that
$x_{u}x_{v}\leq \frac{1-\varepsilon_0}{4\sqrt{h}}.$
Recall that $e(H)=h$ and $H=G_{(\ell)}$.
Let $G':=G_{(\ell)}-\{uv\}$.
Then we obtain
\begin{align}\label{align-01GP}
\rho(G_{(\ell)})&=\mathbf{x}^\top A(G_{(\ell)})\mathbf{x}=\mathbf{x}^\top A(G')\mathbf{x}\!+\!2x_{u}x_{v}
\leq \rho(G')\!+\!\frac{1\!-\!\varepsilon_0}{2\sqrt{h}}.
\end{align}
Since $e(G_{(\ell)})=e(G')+1$,
we have
\begin{align}\label{align-02GP}
\sqrt{e(G_{(\ell)})}-\sqrt{e(G')}=\frac{1}{\sqrt{e(G_{(\ell)})}\!+\!\sqrt{e(G')}}
\geq\frac{1}{2\sqrt{e(G_{(\ell)})}}.
\end{align}
Combining with \eqref{align-01GP}, \eqref{align-02GP} and $e(G')=e(G_{(\ell)})-1$ yields
\begin{align*}
\rho(G')\sqrt{e(G_{(\ell)})}\!-\!\rho(G_{(\ell)})\sqrt{e(G')}
\geq \rho(G_{(\ell)})\big(\sqrt{e(G_{(\ell)})}\!-\!\sqrt{e(G')}\big)\!-\!\frac{1\!-\!\varepsilon_0}{2}
\geq \frac{\rho(G_{(\ell)})}{2\sqrt{e(G_{(\ell)})}}\!-\!\frac{1\!-\!\varepsilon_0}{2}.
\end{align*}
By Lemma \ref{Lem4.2}, we have $\rho(G_{(\ell)})\geq \sqrt{h}$.
It follows that $\rho(G')\sqrt{e(G_{(\ell)})}\!-\!\rho(G_{(\ell)})\sqrt{e(G')}\geq\frac{\varepsilon_0}{2}.$
Note that $e(G')<e(G_{(\ell)})= h.$
Consequently,
\begin{align*}
\Phi(G')\!-\!\Phi(G_{(\ell)})
=\frac{\rho(G')\sqrt{e(G_{(\ell)})}
\!-\!\rho(G_{(\ell)})\sqrt{e(G')}}{\sqrt{e(G')e(G_{(\ell)})}}
\geq \frac{\varepsilon_0}{2h}
=\varepsilon_0\frac{e(G_{(\ell)})-e(G')}{2e(G_{(\ell)})}.
\end{align*}

Let $G''$ be obtained from $G'$ by deleting all isolated vertices.
Since deleting isolated vertices changes neither the number of edges nor the spectral radius, we have
$e(G'')=e(G')$ and $\rho(G'')=\rho(G')$,
and hence
$\Phi(G'')=\Phi(G').$
Moreover, since $e(G'')=e(G')=h-1$ and $h>m/2$, for sufficiently large $m$ we have
\[
(1-\varepsilon_0)h<h-1=e(G'')<h=e(G_{(\ell)}).
\]
Thus, by Definition~\ref{def3.1G}, $G''$ is an
$\varepsilon_0$-dense subgraph of $G_{(\ell)}$.
By Lemma \ref{Lem4.1}, we have
$e(G_{(1)})-e(G_{(\ell)})<\lfloor \varepsilon_0 m\rfloor$.
Since the left-hand side is an integer, it follows that
$e(G_{(1)})-e(G_{(\ell)})\leq \lfloor \varepsilon_0 m\rfloor-1.$
Therefore,
$$e(G_{(1)})-e(G'')=e(G_{(1)})-e(G')=e(G_{(1)})-e(G_{(\ell)})+1
\leq \lfloor \varepsilon_0 m\rfloor.$$
Thus, $G''$ could be taken as a further $\varepsilon_0$-dense subgraph
after $G_{(\ell)}$, while the total number of deleted edges would still
be less than $2\lfloor\varepsilon_0m\rfloor$. This contradicts the
termination of the construction at $G_{(\ell)}$.
This proves the lemma.
\end{proof}

The following lemma provides a lower bound on the eigenvector coordinates
for vertices of relatively small degrees.

\begin{lem}\label{Lem4.4}
For any vertex $u\in V(H)$ with $d_H(u)< \varepsilon_0 h$,
we have $2h x_u^2\geq (1-2\varepsilon_0)d_H(u).$
\end{lem}

\begin{proof}
Suppose to the contrary that there exists a vertex $u_0\in V(H)$ with
$d_H(u_0)<\varepsilon_0 h$ such that
$2h x_{u_0}^2<(1-2\varepsilon_0)d_H(u_0).$
By the definition of $H$, we know that $d_H(u_0)\geq 1$.
Then
\[
x_{u_0}^2
<
\frac{(1-2\varepsilon_0)d_H(u_0)}{2h}
<
\frac12\varepsilon_0(1-2\varepsilon_0)
<\frac{\varepsilon_0}{1+\varepsilon_0}.
\]
 Hence
$1-x_{u_0}^2>\frac{1}{1+\varepsilon_0}.$
Let $H_0=H-{u_0}$, and let $\mathbf{x}'$ be the restriction of $\mathbf{x}$ to
$V(H_0)$. By the Rayleigh quotient,
\[
\rho(H_0)
\geq
\frac{(\mathbf{x}')^\top A(H_0)\mathbf{x}'}
{(\mathbf{x}')^\top \mathbf{x}'}
=
\frac{\sum_{uv\in E(H_0)}2x_u x_v}{1-x_{u_0}^2}.
\]
Thus
$\sum_{uv\in E(H_0)}2x_u x_v \leq \rho(H_0)(1-x_{u_0}^2).$

On the other hand, using the eigenvalue equation at $u_0$, we have
\begin{align*}
\rho
&=\mathbf{x}^\top A(H)\mathbf{x}=\sum_{uv\in E(H_0)}2x_u x_v+2x_{u_0}\sum_{u\in N_H(u_0)}x_u \\
&=\sum_{uv\in E(H_0)}2x_u x_v+2\rho x_{u_0}^2
\leq\rho(H_0)(1-x_{u_0}^2)+2\rho x_{u_0}^2.
\end{align*}
Consequently,
$\rho(H_0)(1-x_{u_0}^2)\geq(1-2x_{u_0}^2)\rho.$
Since $1-x_{u_0}^2>\frac{1}{1+\varepsilon_0}$, we obtain
\[
\rho(H_0)
\geq
\rho\left(1-\frac{x_{u_0}^2}{1-x_{u_0}^2}\right)
\geq
\rho\bigl(1-(1+\varepsilon_0)x_{u_0}^2\bigr).
\]
Let $h_0=e(H_0)=h-d_H(u_0)$. Then
\begin{align}\label{align-10G}
\Phi(H_0)-\Phi(H)
=
\frac{\rho(H_0)\sqrt h-\rho\sqrt{h_0}}{\sqrt{hh_0}}
\geq
\frac{\rho}
{\sqrt{hh_0}}
\left(
\bigl(1-(1+\varepsilon_0)x_{u_0}^2\bigr)\sqrt h-\sqrt{h_0}
\right).
\end{align}
Since
$\sqrt h-\sqrt{h_0}
=\frac{d_H(u_0)}{\sqrt h+\sqrt{h_0}}
\geq\frac{d_H(u_0)}{2\sqrt h},$
we have
\begin{align*}
\bigl(1-(1+\varepsilon_0)x_{u_0}^2\bigr)\sqrt h-\sqrt{h_0}
 =\sqrt h-\sqrt{h_0}-(1+\varepsilon_0)x_{u_0}^2\sqrt h
\geq
\frac{d_H(u_0)}{2\sqrt h}
\left(
1-(1+\varepsilon_0)\frac{2h x_{u_0}^2}{d_H(u_0)}
\right).
\end{align*}
By our assumption,
\[
(1+\varepsilon_0)\frac{2h x_{u_0}^2}{d_H(u_0)}
<
(1+\varepsilon_0)(1-2\varepsilon_0)
=
1-\varepsilon_0-2\varepsilon_0^2
<
1-\varepsilon_0.
\]
Therefore,
$\bigl(1-(1+\varepsilon_0)x_{u_0}^2\bigr)\sqrt h-\sqrt{h_0}
\geq\frac{\varepsilon_0 d_H(u_0)}{2\sqrt h}.$
Combining this with \eqref{align-10G}, and using $\rho\geq \sqrt h$ and
$h_0<h$, we get
$\Phi(H_0)-\Phi(H)
\geq\frac{\varepsilon_0 d_H(u_0)}{2h}.$

Let $H_0^*$ be obtained from $H_0$ by deleting all isolated vertices.
Since deleting isolated vertices changes neither the number of edges nor the spectral radius, we have
$e(H_0^*)=e(H_0)$ and $\rho(H_0^*)=\rho(H_0)$,
and hence
$\Phi(H_0^*)=\Phi(H_0).$
Moreover,
$1\leq e(H)-e(H_0^*)
=d_H(u_0)
<\varepsilon_0 h.$
Thus,
$(1-\varepsilon_0)e(H)<e(H_0^*)<e(H),$
and
$
\Phi(H_0^*)-\Phi(H)
\geq
\frac{\varepsilon_0\bigl(e(H)-e(H_0^*)\bigr)}
     {2e(H)}.
$
Therefore, $H_0^*$ is an $\varepsilon_0$-dense subgraph of $H$.

Finally, by Lemma~\ref{Lem4.1},
$
e(G_{(1)})-e(G_{(\ell)})
<\lfloor\varepsilon_0 m\rfloor.$
Since $H=G_{(\ell)}$, $h\leq m$, and
$e(H)-e(H_0^*)<\varepsilon_0 h
\leq \varepsilon_0 m,$
we have
$e(H)-e(H_0^*)\leq\lfloor\varepsilon_0 m\rfloor.$
Consequently,
$e(G_{(1)})-e(H_0^*)<2\lfloor\varepsilon_0 m\rfloor.$
Thus, $H_0^*$ could be taken as a further $\varepsilon_0$-dense subgraph after
$G_{(\ell)}$, while the total number of deleted edges would still be less than
$2\lfloor\varepsilon_0 m\rfloor$.
This contradicts the termination of the construction at $G_{(\ell)}$.
Therefore, $2h x_u^2\geq(1-2\varepsilon_0)d_H(u)$
for every $u\in V(H)$ with $d_H(u)<\varepsilon_0 h$.
\end{proof}

By \eqref{align-002T}, we have
$N(K_{s,t}^+,H)=o\bigl(h^{(s+t)/2}\bigr).$
Applying Lemma~\ref{thm2.1G} to $H$, we obtain two disjoint subsets
$U_1,U_2\subseteq V(H)$ with $|U_1|\leq |U_2|$ such that
$d(H,K_{U_1,U_2})\leq \varepsilon h.$
We choose such a pair $(U_1,U_2)$ so that $d(H,K_{U_1,U_2})$ is minimized. Then
\begin{align}\label{align-003R}
(1-\varepsilon)h\leq |U_1||U_2|\leq (1+\varepsilon)h.
\end{align}
For a vertex $v\in V(H)$ and  a vertex subset $X\subseteq V(H)$ (possibly $v\notin X$),
we write $N_X(v):=N_H(v)\cap X$ and $d_X(v):=|N_X(v)|$.

\begin{lem}\label{Lem4.5}
Let $R = V(H) \setminus (U_1 \cup U_2)$.
Then the following statements hold:

{\rm (i)} For each $i\in\{1,2\}$ and every vertex $u\in U_i$, we have
$|N_{U_{3-i}}(u)|\ge \frac12 |U_{3-i}|$;

{\rm (ii)}
For every vertex $w\in R$ and each $i\in\{1,2\}$, we have $|N_{U_i}(w)|\leq \frac{1}{2}|U_i|$.
\end{lem}

\begin{proof}
(i) Suppose, for contradiction, that the assertion fails.
We may assume that $i=1$, as the case $i=2$ is analogous.
Then there exists a vertex $u \in U_1$ such that
$|N_{U_2}(u)| < \frac12|U_2|$.

Move $u$ from $U_1$ to $R$, thereby replacing the partition
$(U_1,U_2,R)$ with
$(U_1\setminus\{u\},\,U_2,\,R\cup\{u\})$.
Let
$H_1=K_{U_1,U_2}$ and
$H_2=K_{U_1\setminus\{u\},U_2}$.
The only vertex pairs whose adjacency status differs in $H_1$ and $H_2$
are the pairs $\{u,v\}$ with $v\in U_2$: each such pair is an edge of
$H_1$ but not of $H_2$.
To compute the change in the size of the symmetric difference, we consider
the contribution of each such pair $\{u,v\}$:

\begin{itemize}
    \item If $v\in N_{U_2}(u)$, then $uv$ is an edge of both $H$ and $H_1$,
    but not of $H_2$. Hence, $uv$ contributes $0$ to $d(H,H_1)$ and $1$
    to $d(H,H_2)$, and therefore contributes $+1$ to
    $d(H,H_2)-d(H,H_1)$.

    \item If $v\in U_2\setminus N_{U_2}(u)$, then $uv$ is not an edge of
    either $H$ or $H_2$, but is an edge of $H_1$. Hence, $uv$ contributes
    $1$ to $d(H,H_1)$ and $0$ to $d(H,H_2)$, and therefore contributes
    $-1$ to $d(H,H_2)-d(H,H_1)$.
\end{itemize}
Since the status of all other vertex pairs remains unchanged, we have
\[
d(H,H_2)-d(H,H_1)
=
|N_{U_2}(u)|-\bigl(|U_2|-|N_{U_2}(u)|\bigr)
=
2|N_{U_2}(u)|-|U_2|<0.
\]
This contradicts the minimality of $d(H,K_{U_1,U_2})$.
Hence, we must have $|N_{U_{3-i}}(u)| \ge \frac{1}{2}|U_{3-i}|$ for $i=1,2$.

\medskip

(ii) We proceed by contradiction.
We may assume that
$|N_{U_2}(w)|>\frac{1}{2}|U_2|$
for some vertex $w\in R$, as the case involving $U_1$ is analogous.

Move $w$ from $R$ to $U_1$, thereby replacing the partition
$(U_1,U_2,R)$ with
$(U_1\cup\{w\},\,U_2,\,R\setminus\{w\})$.
Let
$H_1=K_{U_1,U_2}$ and
$H_3=K_{U_1\cup\{w\},U_2}$.
The only vertex pairs whose adjacency status differs in $H_1$ and $H_3$
are the pairs $\{w,v\}$ with $v\in U_2$: each such pair is an edge of
$H_3$ but not of $H_1$.
To compute the change in the size of the symmetric difference, we consider
the contribution of each such pair $\{w,v\}$:
\begin{itemize}
    \item If $v\in N_{U_2}(w)$, then $wv$ is an edge of both $H$ and $H_3$,
    but not of $H_1$. Hence, $wv$ contributes $1$ to $d(H,H_1)$ and $0$
    to $d(H,H_3)$, and therefore contributes $-1$ to
   $d(H,H_3)-d(H,H_1)$.

    \item If $v\in U_2\setminus N_{U_2}(w)$, then $wv$ is not an edge of
    either $H$ or $H_1$, but is an edge of $H_3$. Hence, $wv$ contributes
    $0$ to $d(H,H_1)$ and $1$ to $d(H,H_3)$, and therefore contributes
    $+1$ to $d(H,H_3)-d(H,H_1)$.
\end{itemize}
All other vertex pairs make the same contribution to $d(H,H_1)$ and
$d(H,H_3)$. Consequently,
\begin{align*}
d(H,H_3)-d(H,H_1)
=\bigl(|U_2|-|N_{U_2}(w)|\bigr)-|N_{U_2}(w)|
=|U_2|-2|N_{U_2}(w)|
<0.
\end{align*}
This contradicts the choice of $(U_1,U_2)$ as a pair minimizing
$d(H,K_{U_1,U_2})$. By symmetry, the same argument applies if
$|N_{U_1}(w)|>\frac{1}{2}|U_1|$.
Therefore, for every vertex $w\in R$ and each $i\in\{1,2\}$,
we have
$|N_{U_i}(w)|\leq \frac{1}{2}|U_i|$.
\end{proof}

\begin{lem}\label{Lem4.6}
For each $i\in \{1,2\}$ and each integer $j$ with $1\leq j\leq 10$, define $$S_i^{(j)}=\{u\in U_i:|N_{U_{3-i}}(u)|\leq (1-j\varepsilon_2)|U_{3-i}|\}.$$
Then $|S_i^{(j)}|\leq \varepsilon_2 |U_i|$.
\end{lem}

\begin{proof}
Fix an integer $j$ with $1\leq j\leq10$.
We prove only the case
$i=1$, since the case $i=2$ is symmetric.
Suppose to the contrary that $|S_1^{(j)}|>\varepsilon_2 |U_1|.$
By the definition of $S_1^{(j)}$, every vertex $u\in S_1^{(j)}$ misses at least
$j\varepsilon_2 |U_2|$ vertices in $U_2$.
Summing over all vertices in $S_1^{(j)}$,
we find that the number of non-edges between $U_1$ and $U_2$ is at least
$$
\sum_{u\in S_1^{(j)}}\bigl(|U_2|-|N_{U_2}(u)|\bigr)
\geq j\varepsilon_2 |U_2|\, |S_1^{(j)}|
> j\varepsilon_2^2 |U_1||U_2|.
$$
Each such non-edge belongs to the symmetric difference between $H$ and $K_{U_1,U_2}$.
We conclude that
$d(H,K_{U_1,U_2})> j\varepsilon_2^2 |U_1||U_2|.$
By \eqref{align-003R},
we obtain
$d(H,K_{U_1,U_2})> j\varepsilon_2^2(1-\varepsilon)h.$
Since $\varepsilon\ll \varepsilon_2$, we have $d(H,K_{U_1,U_2})>\varepsilon h,$
contradicting $d(H,K_{U_1,U_2})\leq \varepsilon h$.
Thus
$|S_1^{(j)}|\leq \varepsilon_2 |U_1|.$
The same argument gives
$|S_2^{(j)}|\leq \varepsilon_2 |U_2|.$
This completes the proof of Lemma \ref{Lem4.6}.
\end{proof}

For each $i\in\{1,2\}$ and each integer $j$ with $1\leq j\leq10$,
let $U_i^{(j)}:=U_i\setminus S_i^{(j)}$.
Let $H[S,T]$ be the bipartite subgraph on the vertex set $S\cup T$
which consists of all edges with one endpoint in $S$ and the other in $T$.
For brevity, we write $E(S)=E(H[S])$, $E(S,T)=E(H[S,T])$, $e(S)=e(H[S])$ and $e(S,T)=e(H[S,T])$.

\begin{lem}\label{Lem4.7B}
Let $j$ be an integer with $1\leq j\leq 10$, $u_1u_2 \in E(H)$ satisfying $|N_H(u_1)\cap N_H(u_2)\cap U_2^{(j)}| \geq \varepsilon_2|U_2|$,
and let $\mathcal{F}(u_1u_2)$ denote the family of all copies $F$ of $K_{s,t}^{+}$ in $H$ such that $u_1u_2\in E(F)$ and $F-\{u_1u_2\}\cong K_{s,t}$.
Then there exists a positive constant $c_{s,t,\varepsilon_2}$, depending only on $s$, $t$ and $\varepsilon_2$, such that
\[
|\mathcal{F}(u_1u_2)| \ge c_{s,t,\varepsilon_2} |U_2| h^{\frac{s+t-3}{2}}.
\]
\end{lem}

\begin{proof}
Let $L:=N_H(u_1)\cap N_H(u_2)\cap U_2^{(j)}.$
By assumption, $|L|\geq \varepsilon_2 |U_2|$.
Since $|U_2|\geq \sqrt{(1-\varepsilon)h}$, we have
$|U_2|\to\infty$ as $h\to\infty$. Thus, since $t$ is fixed and
$|L|\geq\varepsilon_2|U_2|$, there exists a constant
$\alpha_{t,\varepsilon_2}>0$ such that
$\binom{|L|}{t}\geq \alpha_{t,\varepsilon_2}|U_2|^t$.
We first choose a $t$-set $B\subseteq L$. Then every vertex of $B$ is adjacent to both $u_1$ and $u_2$.

For this fixed $B$, put
$C_B:=\bigcap_{y\in B}N_{U_1}(y).$
Since $B\subseteq U_2^{(j)}$, every $y\in B$ satisfies
$
|N_{U_1}(y)|\geq (1-j\varepsilon_2)|U_1|.
$
Hence, by the union bound,
$
|C_B|\geq (1-jt\varepsilon_2)|U_1|.
$

Now we claim  that $|U_1|\geq s$.
Suppose to the contrary that $|U_1|\leq s-1$, and put
$a:=|U_1|$. By our choice of
$U_1,U_2$, we have
$d(H,K_{U_1,U_2})\leq \varepsilon h.$
Thus, for sufficiently large $h$, Lemma~\ref{lem2.7G} {\rm (ii)} applied
to $H$ with $U=U_1$ and $V=U_2$ gives
$\rho\leq g_a(h).$
For fixed $h$ sufficiently large, the function
$f(x):=\frac{x-1+\sqrt{4h-x^2+1}}{2}$
is increasing on $1\leq x\leq s-1$, since
$f'(x)=\frac12(1-\frac{x}{\sqrt{4h-x^2+1}})>0.$
Hence, as $a\leq s-1$, we get
$\rho\leq g_a(h)\leq g_{s-1}(h),$
contradicting Lemma~\ref{Lem4.2}, which states that
$\rho>g_{s-1}(h)$. Therefore, $|U_1|\geq s$.

We now estimate the number of ways to choose the remaining $s-2$ vertices in the part of size $s$. Since $\varepsilon_2$ is sufficiently small, the above lower bound on $|C_B|$ implies the following. If $|U_1|\geq 100(s+t)$, then
$
|C_B|-2\geq \frac{1}{3}|U_1|.
$
If $s\leq |U_1|<100(s+t)$,
then $jt\varepsilon_2 |U_1|\leq jt\varepsilon_2 \cdot 100(s+t)<1$, and hence $|C_B|>|U_1|-1.$
Since $C_B\subseteq U_1$ and $|C_B|$ is an integer,
it follows that $|C_B|=|U_1|$.
Hence, the number of choices for
$A'\subseteq C_B\setminus\{u_1,u_2\}$ of size $s-2$
is $$\binom{|C_B\setminus \{u_1,u_2\}|}{s-2}\geq \binom{|C_B|-2}{s-2}.$$

We claim that there exists a constant $\beta_{s,t}>0$, depending only
on $s$ and $t$, such that
\begin{align}\label{alig-002}
\binom{|C_B|-2}{s-2}
\geq \beta_{s,t}|U_1|^{s-2}.
\end{align}
Indeed, suppose first that $|U_1|\geq 100(s+t)$. Since
$|C_B|-2\geq \frac{1}{3}|U_1|\geq 2(s-2)$
and $s$ is fixed, it follows that
\begin{align*}
\binom{|C_B|-2}{s-2}
\geq
\frac{\bigl((|C_B|-2)/2\bigr)^{s-2}}{(s-2)!}
\geq
\frac{1}{6^{s-2}(s-2)!}|U_1|^{s-2}.
\end{align*}
Here we have used the fact that $|C_B|-2$ is sufficiently large
compared with $s-2$.

Now suppose that $s\leq |U_1|<100(s+t)$.
In this case, $|C_B|=|U_1|$, and hence
$\binom{|C_B|-2}{s-2}
=\binom{|U_1|-2}{s-2}$.
Since $|U_1|$ ranges over only finitely many integers, the constant
\[
\gamma_{s,t}
:=
\min_{\substack{k\in\mathbb{Z}\\
s\leq k<100(s+t)}}
\frac{\binom{k-2}{s-2}}{k^{s-2}}
\]
is positive. Therefore,
$\binom{|C_B|-2}{s-2}\geq\gamma_{s,t}|U_1|^{s-2}$.

Thus, by taking
\[
\beta_{s,t}
:=
\min\left\{
\frac{1}{6^{s-2}(s-2)!},
\gamma_{s,t}
\right\},
\]
we obtain \eqref{alig-002} in both cases.
Consequently, after fixing $B$, there are at least
$\beta_{s,t}|U_1|^{s-2}$ choices for $A'$.

For each pair $(A',B)$, let
$A:=A'\cup\{u_1,u_2\}$.
Then the edges between $A$ and $B$, together with the edge
$u_1u_2$, form a copy of $K_{s,t}^{+}$ in $H$.
Note that the above construction may count the same copy of
$K_{s,t}^{+}$ several times. However, the multiplicity is bounded by
a constant depending only on $s$ and $t$. Indeed, for a fixed copy
$F\in \mathcal{F}(u_1u_2)$, the possible choices of the set $B$ are
subsets of $V(F)\setminus\{u_1,u_2\}$ of size $t$. Hence the number of
possible choices of $B$ is at most
$\binom{s+t}{t}$.
Once $B$ is fixed, the set $A'$ is uniquely determined as
$A'=V(F)\setminus(B\cup\{u_1,u_2\})$.
Therefore, each copy in $\mathcal{F}(u_1u_2)$ is counted at most
$\binom{s+t}{t}$ times in our construction.
Consequently,
\[
|\mathcal{F}(u_1u_2)|
\ge
\frac{\alpha_{t,\varepsilon_2}\beta_{s,t}}
{\binom{s+t}{t}}
|U_2|^t|U_1|^{s-2}.
\]
Since $|U_1|\leq |U_2|$ and
$|U_1||U_2|\geq (1-\varepsilon)h$,
we have
$|U_2|\geq \sqrt{|U_1||U_2|}\geq \sqrt{(1-\varepsilon)h}$.
Combining these with $t+1\geq s\geq 3$, we obtain
\begin{align*}
|U_2|^t|U_1|^{s-2}
&=|U_2|^{t-s+2}(|U_1||U_2|)^{s-2}
=|U_2|\cdot |U_2|^{t-s+1}(|U_1||U_2|)^{s-2}\\
&\geq
|U_2|\cdot \bigl((1-\varepsilon)h\bigr)^{\frac{t-s+1}{2}}
\cdot \bigl((1-\varepsilon)h\bigr)^{s-2}
\geq |U_2|\big(\frac{h}{2}\big)^{\frac{s+t-3}{2}}.
\end{align*}
Hence
$
|\mathcal{F}(u_1u_2)|
\geq
c_{s,t,\varepsilon_2}|U_2|h^{\frac{s+t-3}{2}}
$
for some positive constant $c_{s,t,\varepsilon_2}$ depending only on $s$, $t$  and $\varepsilon_2$.
 This completes the proof.
\end{proof}

\begin{lem}\label{LEM4.8B}
We have $|U_1|\geq 30s^2.$
\end{lem}

\begin{proof}
Suppose to the contrary that $|U_1|<30s^2$.
It follows that $|U_2|\geq \frac{(1-\varepsilon)h}{|U_1|}\geq \frac{h}{60s^2}$.
By Lemma \ref{Lem4.6}, we have
\[
|S_1^{(8)}|
\leq \varepsilon_2 |U_1|
< 30s^2\varepsilon_2
< 1.
\]
Since $|S_1^{(8)}|$ is an integer, we have $S_1^{(8)}=\varnothing$, and hence
$U_1^{(8)}=U_1$.
If $e(U_1^{(8)})>0$, we may choose an edge $u_1u_2\in E(U_1^{(8)})$.
Then $|N_{U_2}(u_i)|\geq (1-8\varepsilon_2)|U_2|$ for each $i\in \{1,2\}$.
Consequently, this edge $u_1u_2$ satisfy
$$|N_H(u_1)\cap N_H(u_2)\cap U_2^{(8)}|
\geq |N_{U_2}(u_1)|+|N_{U_2}(u_2)|-|U_2|-|S_2^{(8)}|
 \geq (1-17\varepsilon_2)|U_2|>\varepsilon_2 |U_2|.$$
Then by Lemma \ref{Lem4.7B} and $|U_2|\geq \frac{h}{60s^2}$, we have
$$|\mathcal{F}(u_1u_2)|
\geq c_{s,t,\varepsilon_2} |U_2| h^{\frac{s+t-3}{2}}
\geq \frac{1}{60s^2}c_{s,t,\varepsilon_2}h^{\frac{s+t-1}{2}},$$
which contradicts \eqref{align-002T} as $\varepsilon_1\ll \varepsilon_2$.
Consequently, we must have $e(U_1^{(8)})=0$.
Since $U_1^{(8)}=U_1$, this immediately implies $e(U_1)=0$.

Since $|U_1|<30s^2$ and $s$ is fixed, there are only finitely
many possible values of $|U_1|$. Hence Lemma~\ref{lem2.7G}(i)
applies uniformly for all sufficiently large $h$. Since
$e(U_1)=0$, we obtain
$\rho\leq\sqrt h+O_s(\frac1{\sqrt h}).$
On the other hand,
Lemma~\ref{Lem4.2} yields
$$ \rho > g_{s-1}(h) = \sqrt h+\frac{s-2}{2} + O_s\left(\frac1{\sqrt h}\right),$$
which leads to a contradiction.
We therefore conclude that
$|U_1|\geq 30s^2.$
\end{proof}

\begin{lem}\label{Lem4.9C}
For every $w\in V(H)$, there exists some $i\in\{1,2\}$ such that
$|N_{U_i}(w)|\le 6\varepsilon_2 |U_i|.$
Furthermore,
for each $i\in \{1,2\}$ and every $u\in U_i$, we have $|N_{U_i}(u)|\le 6\varepsilon_2 |U_i|$.

\end{lem}

\begin{proof}
Suppose, for contradiction, that there exists a vertex $w\in V(H)$
such that $|N_{U_i}(w)|>6\varepsilon_2|U_i|$
for both $i\in\{1,2\}$.
Since $U_1^{(4)}=U_1\setminus S_1^{(4)}$ and
$|S_1^{(4)}|\leq\varepsilon_2|U_1|$, we have
\[
\begin{aligned}
|N_{U_1^{(4)}}(w)|
\geq |N_{U_1}(w)|-|S_1^{(4)}|
>6\varepsilon_2|U_1|-\varepsilon_2|U_1|
=5\varepsilon_2|U_1|.
\end{aligned}
\]

Fix any vertex $u\in N_{U_1^{(4)}}(w)$.
Let
$L_u:=N_H(w)\cap N_H(u)\cap U_2^{(4)}$.
Since $|U_2\setminus U_2^{(4)}|=|S_2^{(4)}|\leq\varepsilon_2|U_2|$, we have
\[
\begin{aligned}
|N_{U_2^{(4)}}(w)|
\geq |N_{U_2}(w)|-|S_2^{(4)}|
>6\varepsilon_2|U_2|-\varepsilon_2|U_2|
=5\varepsilon_2|U_2|.
\end{aligned}
\]
Moreover, since $u\in U_1^{(4)}$, we have
$|N_{U_2}(u)|>(1-4\varepsilon_2)|U_2|$,
and hence
$|U_2\setminus N_{U_2}(u)|<4\varepsilon_2|U_2|$.
Therefore,
\[
\begin{aligned}
|L_u|
=|N_{U_2^{(4)}}(w)\cap N_{U_2}(u)|
\geq |N_{U_2^{(4)}}(w)|
      -|U_2\setminus N_{U_2}(u)|
>\bigl(5\varepsilon_2-4\varepsilon_2\bigr)|U_2|
=\varepsilon_2|U_2|.
\end{aligned}
\]
Clearly, $N_{U_2^{(4)}}(w)\cap N_{U_2}(u)=N_H(w)\cap N_H(u)\cap U_2^{(4)}$.
Thus, the edge $wu$ satisfies the hypothesis of Lemma~\ref{Lem4.7B}.

Recall that $\mathcal{F}(wu)$ denotes the family of all copies $F$ of
$K_{s,t}^{+}$ in $H$ such that $wu\in E(F)$ and
$F-\{wu\}\cong K_{s,t}$.
We first note that, in any copy of $K_{s,t}^{+}$, the added edge is
the unique edge whose deletion results in a copy of $K_{s,t}$.
Indeed, deleting the added edge clearly gives $K_{s,t}$. On the other
hand, if a cross-edge of the underlying $K_{s,t}$ is deleted, then the
added edge remains in the graph. Since $s\ge3$ and $t\ge2$, the two
endpoints of the added edge still have a common neighbor, and hence the
resulting graph contains a triangle. Therefore, it cannot be
isomorphic to the bipartite graph $K_{s,t}$.

We now show that the families $\mathcal{F}(wu)$, where
$u\in N_{U_1^{(4)}}(w)$, are pairwise disjoint. Suppose, to the
contrary, that there exist two distinct vertices
$u,u'\in N_{U_1^{(4)}}(w)$ and a copy
$F\in\mathcal{F}(wu)\cap\mathcal{F}(wu')$.
Then both $wu$ and $wu'$ are edges of $F$ whose deletion produces a
copy of $K_{s,t}$. By the uniqueness of the added edge proved above,
we must have $wu=wu'$,
which implies $u=u'$, a contradiction. Hence,
$\mathcal{F}(wu)\cap\mathcal{F}(wu')=\varnothing$
for all distinct $u,u'\in N_{U_1^{(4)}}(w)$.

By Lemma~\ref{Lem4.7B}, for each $u \in  N_{U_1^{(4)}}(w)$, we have
$|\mathcal{F}(w u)| \ge c_{s,t,\varepsilon_2} |U_2| h^{\frac{s+t-3}{2}}.$
Summing over all choices of $u \in N_{U_1^{(4)}}(w)$,
 we obtain
\begin{align*}
N(K_{s,t}^+, H) \ge \sum_{u \in N_{U_1^{(4)}}(w)} |\mathcal{F}(w u)|
> \varepsilon_2 |U_1| \cdot c_{s,t,\varepsilon_2} |U_2| h^{\frac{s+t-3}{2}}
= c_{s,t,\varepsilon_2} \varepsilon_2 |U_1||U_2| h^{\frac{s+t-3}{2}}.
\end{align*}
Since $|U_1||U_2| \ge (1-\varepsilon)h$ and $\varepsilon$ is sufficiently small, we get
$N(K_{s,t}^+, H) > c_{s,t,\varepsilon_2} \varepsilon_2 (1-\varepsilon) h^{\frac{s+t-1}{2}}.$
However, by \eqref{align-7G} we have $2^{\frac{s+t-1}{2}}\varepsilon_1 < c_{s,t,\varepsilon_2} \varepsilon_2 (1-\varepsilon)$, and hence
$N(K_{s,t}^+, H) > 2^{\frac{s+t-1}{2}}\varepsilon_1 h^{\frac{s+t-1}{2}},$
which contradicts \eqref{align-002T}.
Therefore, for every $w\in V(H)$, there exists some $i\in\{1,2\}$ such that
$|N_{U_i}(w)|\le 6\varepsilon_2 |U_i|.$

For each $i\in \{1,2\}$ and every vertex $u\in U_i$,
by Lemma~\ref{Lem4.5} (i), we have
$|N_{U_{3-i}}(u)|\geq \frac{1}{2}|U_{3-i}|.$
Since $\varepsilon_2\ll 1$, it follows that
$|N_{U_{3-i}}(u)|\geq \frac{1}{2}|U_{3-i}|>6\varepsilon_2|U_{3-i}|.$
Applying the first part of the Lemma to the vertex $u$, there exists some
$j\in\{1,2\}$ such that
$|N_{U_j}(u)|\leq 6\varepsilon_2|U_j|.$
The above inequality for $N_{U_{3-i}}(u)$ shows that $j\neq 3-i$, and hence we
must have $j=i$. Therefore,
$|N_{U_i}(u)|\leq 6\varepsilon_2|U_i|.$
\end{proof}

\begin{lem}\label{Lem4.10C}
 The following assertions hold:

{\rm (i)} Let $i\in \{1,2\}$.
For any two vertices $u_1\in U_i$ and $u_2\in U_i^{(8)}$,
     we have $x_{u_1}-x_{u_2}\leq 16\varepsilon_2 x_{u^*}$.
Furthermore, for any two vertices $u_1,u_2\in U_i^{(8)}$,
     we have $|x_{u_1}-x_{u_2}|\leq 16\varepsilon_2 x_{u^*}$.

{\rm (ii)} For every $u\in U_i^{(8)}$, we have
   $x_u^2\geq \frac{(1-30\varepsilon_2)|U_{3-i}|}{2h}.$

{\rm (iii)} For every $i\in \{1,2\}$, we have
 $\frac12-20\varepsilon_2\leq \sum_{u\in U_i^{(8)}}x_u^2\leq \sum_{u\in U_i}x_u^2 \leq \frac12+20\varepsilon_2$.
\end{lem}

\begin{proof}
(i)
Let $i\in \{1,2\}$ and let $u_1\in U_i$, $u_2\in U_i^{(8)}$.
If $x_{u_1}< x_{u_2}$, then there is nothing to do.
It remains to consider that $x_{u_1}\ge x_{u_2}$.
Applying the eigenvalue equation twice yields
\[
\rho^2(x_{u_1}-x_{u_2}) = \sum_{v\in N_H(u_1)}\sum_{z\in N_H(v)}x_z - \sum_{v\in N_H(u_2)}\sum_{z\in N_H(v)}x_z.
\]
Observe that the terms for $v \in N_H(u_1) \cap N_H(u_2)$ cancel out.
By the non-negativity of the eigenvector $x$,
discarding the negative terms in the difference yields
\[
\rho^2(x_{u_1}-x_{u_2}) \le \sum_{v \in N_H(u_1) \setminus N_H(u_2)} \sum_{z \in N_H(v)} x_z.
\]
Let $j=3-i$.
We estimate the double sum
$\sum_{v \in N_H(u_1) \setminus N_H(u_2)} \sum_{z \in N_H(v)} x_z$
by considering the following cases:

\begin{itemize}
    \item Suppose that $v \in N_{U_j}(u_1)$ and $z \in N_{U_i}(v)$.
Since $|N_{U_j}(u_2)|\ge (1-8\varepsilon_2)|U_j|$,
we have $|U_j\setminus N_{U_j}(u_2)|\le 8\varepsilon_2|U_j|$.
This yields that
\[
\sum_{v\in N_{U_j}(u_1)\setminus N_{U_j}(u_2)} \sum_{z\in N_{U_i}(v)}x_z
\le |U_j\setminus N_{U_j}(u_2)| \cdot |U_i|x_{u^*}
\le 8\varepsilon_2|U_1||U_2|x_{u^*}.
\]

\item
Suppose that $v \in N_{U_i}(u_1)$ and $z \in N_{U_j}(v)$.
Since $u_1\in U_i$, we get
$|N_{U_j}(u_1)|\geq \frac12|U_j|$.
By Lemma~\ref{Lem4.9C}, there exists $k\in\{1,2\}$ such that
$|N_{U_k}(u_1)|\le6\varepsilon_2|U_k|$.
Since $6\varepsilon_2<\frac12$, the index $k$ cannot be $j$.
Hence $k=i$, and therefore
$|N_{U_i}(u_1)|\le6\varepsilon_2|U_i|$.
Consequently,
\[
\sum_{v\in N_{U_i}(u_1)\setminus N_{U_i}(u_2)} \sum_{z\in N_{U_j}(v)}x_z
\le |N_{U_i}(u_1)| \cdot |U_j|x_{u^*}
\le 6\varepsilon_2|U_1||U_2|x_{u^*}.
\]

\item
It remains to consider the contribution from edges outside $K_{U_1,U_2}$. Since each such edge belongs to the symmetric difference $E(H)\triangle E(K_{U_1,U_2})$, we obtain
\[
\sum_{v\in N_{H}(u_1)\setminus N_{H}(u_2)} \sum_{\substack{z\in N_H(v)\\ vz\notin E(K_{U_1,U_2})}} x_z
\leq \sum_{v\in V(H)} \sum_{\substack{z\in N_H(v)\\ vz\notin E(K_{U_1,U_2})}} x_z
\le 2d(H,K_{U_1,U_2})x_{u^*},
\]
where the factor of $2$ accounts for the two possible orientations of each edge.
\end{itemize}

Summing these three parts yields:
\[
\rho^2(x_{u_1}-x_{u_2}) \le 14\varepsilon_2|U_1||U_2|x_{u^*} + 2d(H,K_{U_1,U_2})x_{u^*}.
\]
Using the bounds $\rho^2\ge h$, $|U_1||U_2|\le (1+\varepsilon)h$, and $d(H,K_{U_1,U_2})\le \varepsilon h$, we have
\[
h(x_{u_1}-x_{u_2}) \le \rho^2(x_{u_1}-x_{u_2}) < \big(14\varepsilon_2(1+\varepsilon) + 2\varepsilon\big)h x_{u^*} < 16\varepsilon_2 h x_{u^*},
\]
where the last inequality holds since $\varepsilon \ll \varepsilon_2$.
Thus, $x_{u_1}-x_{u_2}\leq 16\varepsilon_2 x_{u^*}$.

Finally, let $u_1,u_2\in U_i^{(8)}$.
Since $U_i^{(8)}\subseteq U_i$, applying the first assertion
to the ordered pairs $(u_1,u_2)$ and $(u_2,u_1)$, respectively, gives
\[
x_{u_1}-x_{u_2}\leq 16\varepsilon_2 x_{u^*}
\quad\text{and}\quad
x_{u_2}-x_{u_1}\leq 16\varepsilon_2 x_{u^*}.
\]
Therefore,
$|x_{u_1}-x_{u_2}|\leq 16\varepsilon_2 x_{u^*}$,
which proves the second assertion.

\medskip

(ii)
Suppose first that $u\in U_2^{(8)}$.
By the definition of $U_2^{(8)}$, we have
$|N_{U_1}(u)|\ge (1-8\varepsilon_2)|U_1|.$
Moreover, since $d(H,K_{U_1,U_2})\le \varepsilon h$,
the number of edges inside $U_2$ and the number of edges incident with $R$ are together at most $\varepsilon h$. Hence
\[d_H(u)\le |U_1|+d(H,K_{U_1,U_2})\le |U_1|+\varepsilon h.\]
Since $|U_1||U_2|\le (1+\varepsilon)h$ and $|U_1|\le |U_2|$,
we have $|U_1|\le \sqrt{(1+\varepsilon)h}$.
Therefore, for sufficiently large $h$ and $\varepsilon\ll \varepsilon_0$,
we conclude that $d_H(u)<\varepsilon_0 h.$
By Lemma \ref{Lem4.4}, we obtain $2h x_u^2\ge (1-2\varepsilon_0)d_H(u).$
Since $d_H(u)\ge |N_{U_1}(u)|\ge (1-8\varepsilon_2)|U_1|,$
we have
$2h x_u^2\ge (1-2\varepsilon_0)(1-8\varepsilon_2)|U_1|.$
Consequently,
\begin{align}\label{align-002R}
x_u^2 \ge \frac{(1-2\varepsilon_0)(1-8\varepsilon_2)|U_1|}{2h}
\geq (1-10\varepsilon_2)\frac{|U_1|}{2h}.
\end{align}

Suppose then that $u\in U_1^{(8)}$.
By the definition of $U_1^{(8)}$,
we have $|N_{U_2}(u)|\ge (1-8\varepsilon_2)|U_2|.$
Since $|S_2^{(8)}|\le \varepsilon_2|U_2|$, it follows that
$|N_{U_2^{(8)}}(u)|\ge |N_{U_2}(u)|-|S_2^{(8)}|\ge (1-9\varepsilon_2)|U_2|.$
For every $v\in U_2^{(8)}$, by \eqref{align-002R} we have
$x_v^2\ge(1-10\varepsilon_2)\frac{|U_1|}{2h}$.
Using the eigenvalue equation at $u$,
we obtain $$\rho x_u=\sum_{v\in N_H(u)}x_v\ge\sum_{v\in N_{U_2^{(8)}}(u)}x_v.$$
Consequently,
$\rho x_u\ge |N_{U_2^{(8)}}(u)|\sqrt{(1-10\varepsilon_2)\frac{|U_1|}{2h}}.$
Squaring both sides and using $\rho^2\le (1+2\varepsilon)h$
 and $|N_{U_2^{(8)}}(u)|\geq (1-9\varepsilon_2)|U_2|$, we get
\begin{align*}
x_u^2
\geq
\frac{(1-9\varepsilon_2)^2(1-10\varepsilon_2)|U_1||U_2|}{(1+2\varepsilon)h} \cdot \frac{|U_2|}{2h}
\geq (1-30\varepsilon_2)\frac{|U_2|}{2h}.
\end{align*}
where the last inequality holds since $\varepsilon\ll \varepsilon_0\ll \varepsilon_2$.

\medskip

(iii)
Summing over all $v\in U_2^{(8)}$, and using $|U_2^{(8)}|\ge (1-\varepsilon_2)|U_2|$ and (ii), we get
\[
\sum_{v\in U_2^{(8)}}x_v^2
\ge
|U_2^{(8)}|\cdot \frac{(1-30\varepsilon_2)|U_1|}{2h}
\ge
\frac{(1-\varepsilon_2)(1-30\varepsilon_2)|U_1||U_2|}{2h}.
\]
Since $|U_1||U_2|\ge (1-\varepsilon)h,$ we obtain
$\sum_{v\in U_2^{(8)}}x_v^2\ge
\frac{(1-\varepsilon_2)(1-30\varepsilon_2)(1-\varepsilon)}{2}.$
Since $\varepsilon\ll \varepsilon_0\ll \varepsilon_2,$ we finally have
$\sum_{v\in U_2^{(8)}}x_v^2 \geq \frac12-16\varepsilon_2.$

Summing over all $u\in U_1^{(8)}$, and using $|U_1^{(8)}|\ge (1-\varepsilon_2)|U_1|$ and (ii), we get
\[
\sum_{u\in U_1^{(8)}}x_u^2
\ge
|U_1^{(8)}|\cdot (1-30\varepsilon_2)\frac{|U_2|}{2h}
\ge
\frac{(1-\varepsilon_2)(1-30\varepsilon_2)|U_1||U_2|}{2h}
\geq \frac12-20\varepsilon_2.
\]
where the last inequality holds as $|U_1||U_2|\ge (1-\varepsilon)h$.
Therefore, for each $i\in {1,2}$, we have $\sum_{u\in U_i^{(8)}}x_u^2 \geq \frac12-20\varepsilon_2$.
Consequently,
\[
\sum_{u\in U_i}x_u^2 \le 1 - \sum_{u\in U_{3-i}^{(8)}}x_u^2 \le 1 - \Big(\frac12-20\varepsilon_2\Big) = \frac12+20\varepsilon_2,
\]
which completes the proof.
\end{proof}

\begin{lem}\label{Lem4.11C}
We have $R=\varnothing$.
\end{lem}

\begin{proof}
Suppose to the contrary that there exists a vertex $u\in R$.
By the definition of $H$, we have $d_H(u)\geq 1$.
By Lemma~\ref{Lem4.9C}, there exists some $i\in\{1,2\}$ such that
$|N_{U_i}(u)|\le 6\varepsilon_2|U_i|.$
Let $j=3-i$.
Since $u\in R$, Lemma~\ref{Lem4.5} (ii) gives
$|N_{U_j}(u)|\le \frac12|U_j|.$

Since every edge incident with $u$ belongs to
$E(H)\triangle E(K_{U_1,U_2})$, we have
$$
d_H(u)
\le d(H,K_{U_1,U_2})
\le \varepsilon h
<\varepsilon_0 h.
$$
Hence, by Lemma~\ref{Lem4.4},
$2h x_u^2\ge (1-2\varepsilon_0)d_H(u).$
On the other hand, by the eigenvalue equation and the
Cauchy--Schwarz inequality,
$$
\begin{aligned}
\rho^2x_u^2
=\left(\sum_{v\in N_H(u)}x_v\right)^2
\le
d_H(u)\sum_{v\in N_H(u)}x_v^2.
\end{aligned}
$$
By Lemma~\ref{Lem4.2}, we have $\rho^2>h$. Hence,
\begin{align}\label{align-R1}
\sum_{v\in N_H(u)}x_v^2
\ge \frac{\rho^2x_u^2}{d_H(u)}
\ge\frac{1-2\varepsilon_0}{2}.
\end{align}

We next derive an upper bound on the same quantity.
Since
$|U_i^{(8)}|\ge(1-\varepsilon_2)|U_i|$
and
$|N_{U_i}(u)|\le6\varepsilon_2|U_i|,$
we obtain
$|U_i^{(8)}\setminus N_{U_i}(u)|\ge(1-7\varepsilon_2)|U_i|.$
By Lemma \ref{Lem4.10C} (ii), every $v\in U_i^{(8)}$ satisfies
$x_v^2\ge\frac{(1-30\varepsilon_2)|U_j|}{2h}.$
Therefore, using
$|U_1||U_2|\ge(1-\varepsilon)h$, we get
$$
\begin{aligned}
\sum_{v\in U_i^{(8)}\setminus N_{U_i}(u)}x_v^2
&\ge
(1-7\varepsilon_2)|U_i|
\frac{(1-30\varepsilon_2)|U_j|}{2h}
\ge\frac{(1-7\varepsilon_2)(1-30\varepsilon_2)(1-\varepsilon)}{2}
\ge\frac12-19\varepsilon_2,
\end{aligned}
$$
where the last inequality holds since
$\varepsilon\ll\varepsilon_2\ll1$.
Together with Lemma~\ref{Lem4.10C} (iii), this yields
\begin{align}\label{align-R2}
\sum_{v\in N_{U_i}(u)}x_v^2
\le
\sum_{v\in U_i}x_v^2
-
\sum_{v\in U_i^{(8)}\setminus N_{U_i}(u)}x_v^2
\le
\left(\frac12+20\varepsilon_2\right)
-\left(\frac12-19\varepsilon_2\right)
\le
40\varepsilon_2.
\end{align}

Similarly, since
$|N_{U_j}(u)|\le \frac12|U_j|$
and
$|U_j^{(8)}|\ge(1-\varepsilon_2)|U_j|,$
we have
$|U_j^{(8)}\setminus N_{U_j}(u)|
\ge (\frac12-\varepsilon_2)|U_j|.$
Again by Lemma~\ref{Lem4.10C} (ii),
$$
\begin{aligned}
\sum_{v\in U_j^{(8)}\setminus N_{U_j}(u)}x_v^2
\ge
\left(\frac12-\varepsilon_2\right)|U_j|
\frac{(1-30\varepsilon_2)|U_i|}{2h}
\ge
\frac{\left(\frac12-\varepsilon_2\right)
(1-30\varepsilon_2)(1-\varepsilon)}{2}
\ge
\frac14-9\varepsilon_2,
\end{aligned}
$$
where the last inequality follows from
$\varepsilon\ll\varepsilon_2\ll1$.
Hence Lemma~\ref{Lem4.10C} (iii) gives
\begin{align}\label{align-R3}
\sum_{v\in N_{U_j}(u)}x_v^2
&\le
\sum_{v\in U_j}x_v^2
-
\sum_{v\in U_j^{(8)}\setminus N_{U_j}(u)}x_v^2
\le
\left(\frac12+20\varepsilon_2\right)
-
\left(\frac14-9\varepsilon_2\right)
\le
\frac14+30\varepsilon_2.
\end{align}

Moreover, by Lemma~\ref{Lem4.10C} (iii),
\begin{align}\label{align-R4}
\sum_{v\in R}x_v^2
&=
1-\sum_{v\in U_1}x_v^2-\sum_{v\in U_2}x_v^2
\le
1-\sum_{v\in U_1^{(8)}}x_v^2-\sum_{v\in U_2^{(8)}}x_v^2
\le 40\varepsilon_2.
\end{align}

Combining \eqref{align-R2}, \eqref{align-R3}, and
\eqref{align-R4}, we obtain
$$
\begin{aligned}
\sum_{v\in N_H(u)}x_v^2
=
\sum_{v\in N_{U_i}(u)}x_v^2+\sum_{v\in N_{U_j}(u)}x_v^2+\sum_{v\in N_R(u)}x_v^2
\le
40\varepsilon_2+\left(\frac14+30\varepsilon_2\right)+40\varepsilon_2
\le
\frac14+110\varepsilon_2.
\end{aligned}
$$
Since $\varepsilon_0\ll\varepsilon_2\ll1$, we may choose the
parameters sufficiently small so that
$\frac14+110\varepsilon_2<\frac{1-2\varepsilon_0}{2}.$
This contradicts \eqref{align-R1}.
Therefore, $R=\varnothing$.
\end{proof}

\begin{lem}\label{Lem4.12C}
We have $u^*\in U_1^{(8)}\cup U_2^{(8)}$.
Moreover, if $u^*\in U_i^{(8)}$ for some $i\in \{1,2\}$,
then $x_{u^*}^2\leq (1+80\varepsilon_2) \frac{|U_{3-i}|}{2h}$.
\end{lem}

\begin{proof}
Recall that $u^*\in V(H)$ is a vertex such that $x_{u^*}=\max_{v\in V(H)}x_v$.
We first prove that $u^*\notin S_1^{(8)}$.
Suppose otherwise. Then
$|N_{U_2}(u^*)|\le(1-8\varepsilon_2)|U_2|.$
Since $u^*\notin R$, Lemma~\ref{Lem4.5} implies
$|N_{U_2}(u^*)|\geq \frac12|U_2|>6\varepsilon_2|U_2|.$
Applying Lemma~\ref{Lem4.9C}, the small degree direction
cannot be $U_2$, and hence
$|N_{U_1}(u^*)|\le6\varepsilon_2|U_1|.$
Applying the eigenvalue equation twice, we obtain
\[\rho^2x_{u^*}=\sum_{v\in N_H(u^*)}\sum_{z\in N_H(v)}x_z.
\]
Since $x_z\le x_{u^*}$ for every $z\in V(H)$, it suffices to estimate the number of pairs $(v,z)$ appearing in the above sum.

We first consider the contribution from edges between $U_1$ and $U_2$.
Since $|N_{U_2}(u^*)|\le (1-8\varepsilon_2)|U_2|$,
we have
\[
\sum_{v\in N_{U_2}(u^*)}\sum_{z\in N_{U_1}(v)}x_z \le (1-8\varepsilon_2)|U_2|\cdot |U_1|\,x_{u^*},
\]
and
\[
\sum_{v\in N_{U_1}(u^*)} \sum_{z\in N_{U_2}(v)}x_z
\le  6\varepsilon_2|U_1||U_2|x_{u^*}.
\]

It remains to estimate the contribution from edges that do not belong to the complete bipartite graph $K_{U_1,U_2}$.
Such edges are contained in $E(H)\triangle E(K_{U_1,U_2}),$
and hence
\[
\sum_{v\in N_H(u^*)}\sum_{\substack{z\in N_H(v)\\ vz\notin E(K_{U_1,U_2})}}x_z
\le
2d(H,K_{U_1,U_2})x_{u^*}.
\]
The factor $2$ appears because each edge may be counted from both orientations.

Combining the above estimates,
we obtain \[ \rho^2 x_{u^*} \le \bigl((1-8\varepsilon_2)+6\varepsilon_2\bigr) |U_1||U_2|\,x_{u^*} + 2d(H,K_{U_1,U_2})\,x_{u^*}. \]
Cancelling $x_{u^*}$ and using $|U_1||U_2|\le (1+\varepsilon)h $
 and $d(H,K_{U_1,U_2})\le \varepsilon h,$
we deduce that
\[ \rho^2 \le (1-2\varepsilon_2)(1+\varepsilon)h+2\varepsilon h. \]
Since $\varepsilon\ll \varepsilon_2$, the right-hand side is strictly less than $h$, implying that $\rho^2<h.$
This contradicts Lemma \ref{Lem4.2}.
Therefore, $u^*\notin S_1^{(8)}$.
Similarly, we also get that $u^*\notin S_2^{(8)}$.

Since every vertex outside $U_1^{(8)}\cup U_2^{(8)}$ belongs to
$R\cup S_1^{(8)}\cup S_2^{(8)}$, and we have already shown that
$R=\varnothing$, $u^*\notin S_1^{(8)}$, and $u^*\notin S_2^{(8)}$, it follows that
$u^*\in U_1^{(8)}\cup U_2^{(8)}.$

Assume that $u^*\in U_i^{(8)}$ for some $i\in \{1,2\}$.
Let $u_0 \in U_i^{(8)}$ be a vertex that minimizes $x_u$ over $U_i^{(8)}$,
so that $x_{u_0} = \min_{u \in U_i^{(8)}} x_u$.
By Lemma~\ref{Lem4.10C}~(iii),
we have $\sum_{u \in U_i} x_u^2 \le \frac{1}{2} + 20\varepsilon_2$,
which implies that
\[
(1-\varepsilon_2)|U_i| x_{u_0}^2
\le |U_i^{(8)}| x_{u_0}^2
\le \sum_{u \in U_i^{(8)}} x_u^2
\le \frac{1}{2} + 20\varepsilon_2.
\]
Using the relation $(1-\varepsilon)h \le |U_1||U_2| \le (1+\varepsilon)h$, we obtain
\begin{align*}
x_{u_0}^2
&\le \frac{\frac{1}{2}+20\varepsilon_2}{(1-\varepsilon_2)|U_i|}
\le \frac{\frac{1}{2}+20\varepsilon_2}{1-\varepsilon_2} \cdot \frac{|U_{3-i}|}{(1-\varepsilon)h}.
\end{align*}
Applying Lemma \ref{Lem4.10C} (i) with $u_1=u^*$ and $u_2=u_0$ yields that
$x_{u^*}-x_{u_0}\leq 16\varepsilon_2 x_{u^*}$,
which implies that $$x_{u^*}^2\leq (1+17\varepsilon_2)^2x_{u_0}^2
\leq (1+17\varepsilon_2)^2\cdot \left(1+44\varepsilon_2\right) \frac{|U_{3-i}|}{2h}
\leq \left(1+80\varepsilon_2\right) \frac{|U_{3-i}|}{2h}.$$
This completes the proof of Lemma \ref{Lem4.12C}.
\end{proof}

\section{Proof of Theorem  \ref{thm1.1}}\label{sec5A}

In this section, we complete the proof of Theorem \ref{thm1.1}.

\begin{proof}[\textbf{Proof of Theorem \ref{thm1.1}}]
We proceed by contradiction and distinguish two cases.

\medskip
\noindent{{\bf{Case 1.}}} $|U_2|\geq 100|U_1|$.
\medskip

Since $|U_2| \geq 100|U_1|$, we have
$
|U_2|^2 \geq 100|U_1||U_2|
\geq 100(1-\varepsilon)h
$
and
$
100|U_1|^2 \leq |U_1||U_2|
\leq (1+\varepsilon)h.
$
Consequently,
$|U_2| \geq 10\sqrt{(1-\varepsilon)h}$ and
$|U_1| \leq \frac{1}{10}\sqrt{(1+\varepsilon)h}$.
Since $\varepsilon$ is sufficiently small, it follows that
\begin{align}\label{ali-02}
|U_2|\geq 9\sqrt{h}
\quad\text{and}\quad
|U_1|\leq \frac{1}{9}\sqrt{h}.
\end{align}

\begin{claim}\label{CLA5.1}
For every vertex $u\in U_2$, we have
    $x_u\leq \frac{1}{4} x_{u^*}$.
\end{claim}

\begin{proof}
Let $u_0\in U_2^{(8)}$ satisfy
$x_{u_0}=\min_{u\in U_2^{(8)}}x_u$.
By Lemma~\ref{Lem4.10C} (iii) and
$|U_2^{(8)}|\ge(1-\varepsilon_2)|U_2|$, we have
$$
(1-\varepsilon_2)|U_2|x_{u_0}^2
\le
\sum_{u\in U_2^{(8)}}x_u^2
\le
\frac12+20\varepsilon_2.
$$
Since $|U_2|^2 \geq 100|U_1||U_2| \geq 100(1-\varepsilon)h$, it follows that
$$
\begin{aligned}
x_{u_0}^2
\le
\frac{\frac12+20\varepsilon_2}
{(1-\varepsilon_2)|U_2|}
\le
\frac{\frac12+20\varepsilon_2}{1-\varepsilon_2}
\cdot
\frac{|U_2|}{100(1-\varepsilon)h}
\le
\frac1{20}\cdot
\frac{(1-30\varepsilon_2)|U_2|}{2h},
\end{aligned}
$$
where the last inequality holds since
$\varepsilon\ll\varepsilon_2\ll1$.

Choose any $v\in U_1^{(8)}$. By Lemma~\ref{Lem4.10C} (ii) and the maximality of $x_{u^*}$,
$x_{u^*}^2\ge x_v^2
\ge\frac{(1-30\varepsilon_2)|U_2|}{2h}.$
Hence
$x_{u_0}\le\frac1{\sqrt{20}}x_{u^*}.$
Finally, by Lemma~\ref{Lem4.10C} (i), for every
$u\in U_2$,
$$
x_u
\le
x_{u_0}+16\varepsilon_2x_{u^*}
\le
\left(\frac1{\sqrt{20}}+16\varepsilon_2\right)x_{u^*}
\le
\frac14x_{u^*},
$$
provided that $\varepsilon_2$ is sufficiently small.
\end{proof}

\begin{claim}\label{CLA5.2}
We have $u^*\in U_1^{(8)}$.
\end{claim}

\begin{proof}
By Lemma~\ref{Lem4.12C}, we have
$u^*\in U_1^{(8)}\cup U_2^{(8)}$.
Suppose, for contradiction, that $u^*\in U_2^{(8)}$.
By Claim~\ref{CLA5.1}, every $u\in U_2^{(8)}$ satisfies
$x_u\le \frac14 x_{u^*}$.
Taking $u=u^*$, we obtain
$x_{u^*}\le \frac14 x_{u^*}$,
which contradicts $x_{u^*}>0$.
Therefore, $u^*\notin U_2^{(8)}$, and hence
$u^*\in U_1^{(8)}$.
\end{proof}

Let $E'=E(U_1)\setminus E(S_1^{(8)})$.
We claim that $|E'|\leq \varepsilon_2|U_1|$.
Indeed, suppose to the contrary that
$|E'|>\varepsilon_2|U_1|$.
For every edge $e=uv\in E'$, since
$E'=E(U_1)\setminus E(S_1^{(8)})$, at least one endpoint of $e$,
say $u$, belongs to $U_1^{(8)}$. Hence,
by the definition of $U_1^{(8)}$,
$|N_{U_2}(u)|\geq (1-8\varepsilon_2)|U_2|$.
On the other hand, since $v\in U_1$, Lemma \ref{Lem4.5} (i) gives
$|N_{U_2}(v)|\geq \frac12|U_2|$.
Moreover, by Lemma \ref{Lem4.6},
$|S_2^{(8)}|\leq \varepsilon_2|U_2|$.
Therefore,
\[
\begin{aligned}
|N_H(u)\cap N_H(v)\cap U_2^{(8)}|
\geq
|N_{U_2}(u)|+|N_{U_2}(v)|
-|U_2|-|S_2^{(8)}|
\geq
\left(\frac12-9\varepsilon_2\right)|U_2|
>\varepsilon_2|U_2|,
\end{aligned}
\]
where the last inequality holds since $\varepsilon_2$ is sufficiently small.
Thus, every edge $e\in E'$ satisfies the hypothesis of
Lemma \ref{Lem4.7B} with $j=8$.

Furthermore, the families $\mathcal{F}(e)$, $e\in E'$, are pairwise
disjoint, since in every copy of $K_{s,t}^{+}$ the added edge is the
unique edge whose deletion yields a copy of $K_{s,t}$.
Hence, by Lemma \ref{Lem4.7B},
\[
\begin{aligned}
N(K_{s,t}^+,H)
\geq \sum_{e\in E'}|\mathcal{F}(e)|
\geq
|E'|\,c_{s,t,\varepsilon_2}|U_2|
h^{\frac{s+t-3}{2}}
>
\varepsilon_2|U_1|\,
c_{s,t,\varepsilon_2}|U_2|
h^{\frac{s+t-3}{2}}
\geq
2^{\frac{s+t-1}{2}}\varepsilon_1
h^{\frac{s+t-1}{2}},
\end{aligned}
\]
which contradicts \eqref{align-002T}, since
$\varepsilon_1\ll\varepsilon_2$.
Therefore,
$|E'|\leq\varepsilon_2|U_1|$.

Since $|E'|\leq\varepsilon_2 |U_1|$ and
$|U_1^{(8)}|\geq (1-\varepsilon_2)|U_1|$, after removing all
vertices incident with edges of $E'$, at least
$|U_1^{(8)}|-2|E'| \geq (1-\varepsilon_2)|U_1|-2|E'|$
vertices remain. Denote the set of these remaining vertices by
$U_1'$.

For every $u\in U_1'$, we have
$|N_{U_2}(u)|\geq (1-8\varepsilon_2)|U_2|$.
Since
$(1-8\varepsilon_2)|U_2|>\frac{|U_1|}{4}$,
we can greedily choose a matching
$E''\subseteq E(U_1',U_2)$
of size
$\lfloor\frac{|U_1|}{4}\rfloor$.
Let $Q$ be the set of vertices in $U_1'$ not incident with
edges of $E''$. Then $Q$ is independent in $H[U_1]$, and hence
$Q$ is a clique in the complement of $H[U_1]$.

We claim that the complement of $H[U_1]$ contains a matching of size
$|E'|+|E''|.$
Indeed,
\begin{align*}
|Q|-2(|E'|+|E''|)
\geq (1-\varepsilon_2)|U_1|-4|E'|-3|E''|
>(\frac14-5\varepsilon_2){|U_1|}.
\end{align*}
Since $\varepsilon_2\ll 1$, we have
$\frac14-5\varepsilon_2>0.$
Thus
$|Q|\geq 2(|E'|+|E''|).$
Since $Q$ induces a clique in the complement of $H[U_1]$ and
$|Q|\geq 2(|E'|+|E''|)$, there exists a matching
$M\subseteq \binom{Q}{2}\setminus E(H)$
with $|M|=|E'|+|E''|$.
Let $H^\star$ be obtained from $H$ by
deleting all edges of $E'$ and $E''$, and then adding all edges
of $M$.
Clearly,
$|M|=|E'|+|E''|$,
and hence
$e(H^\star)=e(H)=h$.
Furthermore, by \eqref{ali-02},
\begin{align}\label{align-001G}
d(H,H^\star)\leq |M|+|E'|+|E''|\leq |U_1|\leq \frac{1}{9}\sqrt{h}\leq \varepsilon h.
\end{align}
For brevity, write \(\rho^\star=\rho(H^\star)\).
By the Perron-Frobenius theorem, there exists a nonnegative unit eigenvector
$\mathbf{y} = (y_1, \dots, y_n)^\top$ corresponding to $\rho^\star = \rho(H^\star)$.
Let $u^\star$ be a vertex of $H^\star$ that maximizes the entries of $\mathbf{y}$,
i.e., $y_{u^\star} = \max_{v \in V(H^\star)} y_v$.
For a vertex $v\in V(H^\star)$ and a vertex subset $X\subseteq V(H^\star)$ (possibly $v\notin X$),
we write $N^\star_X(v):=N_{H^\star}(v)\cap X$.

\begin{claim}\label{CLA5.3}
We have $\rho^\star>g_{s-1}(h)>\sqrt{h}.$
\end{claim}

\begin{proof}
By Lemma~\ref{Lem4.10C} (i) and Claim~\ref{CLA5.2}, every vertex $u\in U_1^{(8)}$
satisfies $x_u\geq (1-16\varepsilon_2)x_{u^*}.$
By Claim~\ref{CLA5.1}, every vertex $v\in U_2$ satisfies
$x_v\leq \frac14 x_{u^*}.$
Using $\mathbf{x}$ in the Rayleigh quotient for $H^\star$, we get
$$
\begin{aligned}
\rho^\star-\rho
&\geq
2\sum_{uv\in M}x_ux_v
-2\sum_{uv\in E'}x_ux_v
-2\sum_{uv\in E''}x_ux_v  \\
&\geq
2(|E'|+|E''|)
(1-16\varepsilon_2)^2x_{u^*}^2
-2|E'| x_{u^*}^2
-\frac12|E''| x_{u^*}^2\\
&\geq (-64\varepsilon_2|E'|+|E''|)x_{u^*}^2.
\end{aligned}
$$
Since $|U_1|\geq 30s^2$, we have
$|E''|=\lfloor\frac{|U_1|}4\rfloor\geq \frac{|U_1|}5$ and
$|E'|\leq\varepsilon_2|U_1|.$
Therefore, for sufficiently small \(\varepsilon_2>0\),
the previous inequality implies that
$\rho^\star-\rho>0$,
and hence $\rho^\star>\rho$.
Combining this with Lemma~\ref{Lem4.2} and \eqref{equ-001}, we obtain
$\rho^\star>\rho>g_{s-1}(h)>\sqrt{h}$,
as required.
\end{proof}

\begin{claim}\label{CLA5.4}
Let $i\in \{1,2\}$.
For any two vertices $u_1\in U_i$ and $u_2\in U_i^{(8)}$,
     we have $y_{u_1}-y_{u_2}\leq 16\varepsilon_2 y_{u^\star}$.
Furthermore, for any two vertices $u_1,u_2\in U_i^{(8)}$,
     we have $|y_{u_1}-y_{u_2}|\leq 16\varepsilon_2 y_{u^\star}$.
\end{claim}

\begin{proof}
If $y_{u_1} \le y_{u_2}$, then the desired inequality is immediate.
Hence assume that $y_{u_1} > y_{u_2}$.
Applying the eigenvalue equation twice yields
\[
(\rho^\star)^2(y_{u_1}-y_{u_2}) = \sum_{v\in N_{H^\star}(u_1)}\sum_{z\in N_{H^\star}(v)}y_z - \sum_{v\in N_{H^\star}(u_2)}\sum_{z\in N_{H^\star}(v)}y_z.
\]
Observe that the terms for $v \in N_{H^\star}(u_1) \cap N_{H^\star}(u_2)$ cancel out.
By the non-negativity of the eigenvector $\mathbf{y}$,
discarding the negative terms in the difference yields
\[
(\rho^\star)^2(y_{u_1}-y_{u_2}) \le \sum_{v \in N_{H^\star}(u_1) \setminus N_{H^\star}(u_2)} \sum_{z \in N_{H^\star}(v)} y_z.
\]

Let $j=3-i$. Since the only edges between $U_1$ and $U_2$ deleted when passing from $H$ to $H^\star$ are those in the matching $E''$, and no edge between $U_1$ and $U_2$ is added, we have $N^\star_{U_j}(u)\subseteq N_{U_j}(u)$
for every $u\in U_i$. Moreover, since $E''$ is a matching,
$|N_{U_j}(u)\setminus N^\star_{U_j}(u)|\leq 1.$
Therefore,
$$
N^\star_{U_j}(u_1)\setminus N^\star_{U_j}(u_2)
\subseteq
\bigl(N_{U_j}(u_1)\setminus N_{U_j}(u_2)\bigr)
\cup
\bigl(N_{U_j}(u_2)\setminus N^\star_{U_j}(u_2)\bigr).
$$
Since $u_2\in U_i^{(8)}$, we have
$|U_j\setminus N_{U_j}(u_2)|\leq 8\varepsilon_2|U_j|,$
and hence
$|N_{U_j}(u_1)\setminus N_{U_j}(u_2)|\leq 8\varepsilon_2|U_j|.$
Thus, except for at most one vertex, the set
$N^\star_{U_j}(u_1)\setminus N^\star_{U_j}(u_2)$
has size at most $8\varepsilon_2|U_j|$.
The possible exceptional vertex contributes at most $\rho^\star y_{u^\star}$ by the eigenvalue equation.
Consequently,
$$
\begin{aligned}
\sum_{v\in N^\star_{U_j}(u_1)\setminus N^\star_{U_j}(u_2)}
\sum_{z\in N^\star_{U_i}(v)}y_z
\leq
8\varepsilon_2|U_i||U_j|y_{u^\star}
+\rho^\star y_{u^\star}.
\end{aligned}
$$

We next consider
$\sum_{v\in N^\star_{U_i}(u_1)\setminus N^\star_{U_i}(u_2)}\sum_{z\in N^\star_{U_j}(v)}y_z.$
By Lemma $\ref{Lem4.9C}$,
$|N_{U_i}(u_1)|\leq 6\varepsilon_2|U_i|.$
When passing from $H$ to $H^\star$, no edge is added inside $U_2$, while the only edges added inside $U_1$ form the matching $M$. Hence, for either $i=1$ or $i=2$,
$N^\star_{U_i}(u_1)\subseteq N_{U_i}(u_1)\cup B,$
where $|B|\leq1$. Therefore,
$$
\begin{aligned}
\sum_{v\in N^\star_{U_i}(u_1)\setminus N^\star_{U_i}(u_2)}
\sum_{z\in N^\star_{U_j}(v)}y_z
\leq
6\varepsilon_2|U_i||U_j|y_{u^\star}
+\rho^\star y_{u^\star}.
\end{aligned}
$$

Finally, the contribution corresponding to edges outside $K_{U_1,U_2}$ satisfies
$$
d(H^\star,K_{U_1,U_2})
\leq d(H^\star,H)+d(H,K_{U_1,U_2})
\leq2\varepsilon h.
$$
Hence
$$
\sum_{v\in N_{H^\star}(u_1)\setminus N_{H^\star}(u_2)}
\sum_{\substack{z\in N_{H^\star}(v)\\
vz\notin E(K_{U_1,U_2})}}y_z
\leq4\varepsilon h y_{u^\star}.
$$

Combining the above estimates gives
$$
(\rho^\star)^2(y_{u_1}-y_{u_2})
\leq
\left(
14\varepsilon_2|U_1||U_2|
+4\varepsilon h
+2\rho^\star
\right)y_{u^\star}.
$$
Since $e(H^\star)=h$, by Lemma \ref{lem2.8} we have
$\rho^\star\leq\sqrt{2h}.$
Moreover, by Claim $\ref{CLA5.3}$,
$(\rho^\star)^2>h,$
and $|U_1||U_2|\leq(1+\varepsilon)h$. Therefore,
$$
\begin{aligned}
h(y_{u_1}-y_{u_2})
<
(\rho^\star)^2(y_{u_1}-y_{u_2})
\leq
\left(
14\varepsilon_2(1+\varepsilon)
+4\varepsilon
+\frac{2\sqrt{2}}{\sqrt h}
\right)
h y_{u^\star}
< 16\varepsilon_2h y_{u^\star},
\end{aligned}
$$
where the last inequality holds for sufficiently large $h$, since
$\varepsilon\ll\varepsilon_2$.
Thus,
$y_{u_1}-y_{u_2}\leq16\varepsilon_2y_{u^\star}.$

Finally, let $u_1,u_2\in U_i^{(8)}$.
Since $U_i^{(8)}\subseteq U_i$, applying the first assertion
to the ordered pairs $(u_1,u_2)$ and $(u_2,u_1)$, respectively, gives
\[
y_{u_1}-y_{u_2}\leq 16\varepsilon_2 y_{u^\star}
\quad\text{and}\quad
y_{u_2}-y_{u_1}\leq 16\varepsilon_2 y_{u^\star}.
\]
Therefore,
$|y_{u_1}-y_{u_2}|\leq 16\varepsilon_2 y_{u^\star}$,
which proves the second assertion.
\end{proof}

\begin{claim}\label{CLA5.5}
We have $u^\star\in U_1^{(8)}\cup U_2^{(8)}.$
\end{claim}

\begin{proof}
By Lemma~\ref{Lem4.11C}, we have $R=\varnothing$, and hence
$u^\star\in U_1\cup U_2$.
Suppose, for contradiction, that
$u^\star\in S_i^{(8)}$ for some $i\in\{1,2\}$, and put $j=3-i$.
By the definition of $S_i^{(8)}$,
$|N_{U_j}(u^\star)|\leq(1-8\varepsilon_2)|U_j|.$
Moreover, since $u^\star\in U_i$, Lemma~\ref{Lem4.9C} yields
$|N_{U_i}(u^\star)|\leq6\varepsilon_2|U_i|.$

We claim that the corresponding upper bounds continue to hold in
$H^\star$.
Since no edge between $U_1$ and $U_2$ is added when passing from
$H$ to $H^\star$, we have
$|N^\star_{U_j}(u^\star)|\leq|N_{U_j}(u^\star)|\leq
(1-8\varepsilon_2)|U_j|.$
It remains to estimate the number of neighbors of $u^\star$ in $U_i$.

If $i=2$, then no edge is added inside $U_2$, and hence
$|N^\star_{U_2}(u^\star)|\leq|N_{U_2}(u^\star)|\leq6\varepsilon_2|U_2|.$
If $i=1$, then $u^\star\in S_1^{(8)}$, whereas every edge of the
added matching $M$ has both endpoints in
$Q\subseteq U_1^{(8)}$.
Thus, no edge of $M$ is incident with $u^\star$, and therefore
$|N^\star_{U_1}(u^\star)|\leq|N_{U_1}(u^\star)|\leq6\varepsilon_2|U_1|.$
Consequently, in either case,
$|N^\star_{U_i}(u^\star)|\leq6\varepsilon_2|U_i|.$

Applying the eigenvalue equation twice at $u^\star$, we obtain
$$
(\rho^\star)^2y_{u^\star}
=
\sum_{v\in N_{H^\star}(u^\star)}
\sum_{z\in N_{H^\star}(v)}y_z.
$$
The contribution arising from edges between $U_1$ and $U_2$ is at most
$\bigl((1-8\varepsilon_2)+6\varepsilon_2\bigr)|U_1||U_2|y_{u^\star}.$
Furthermore,
$$
d(H^\star,K_{U_1,U_2})
\leq
d(H^\star,H)+d(H,K_{U_1,U_2})
\leq
2\varepsilon h.
$$
Hence the contribution arising from edges outside
$K_{U_1,U_2}$ is at most
$4\varepsilon h\,y_{u^\star}.$
It follows that
$$
\begin{aligned}
(\rho^\star)^2
\leq
(1-2\varepsilon_2)|U_1||U_2|
+4\varepsilon h
\leq
\left(
(1-2\varepsilon_2)(1+\varepsilon)
+4\varepsilon
\right)h
&<h,
\end{aligned}
$$
where the last inequality follows from
$\varepsilon\ll\varepsilon_2$.
This contradicts Claim~\ref{CLA5.3}, which gives
$\rho^\star>\sqrt h$.
Therefore,
$u^\star\in U_1^{(8)}\cup U_2^{(8)}.$
\end{proof}

\begin{claim}\label{CLA5.6}
We have $u^\star\in U_1^{(8)}.$
\end{claim}

\begin{proof}
By Claim \ref{CLA5.5},
we have $u^\star \in U_1^{(8)} \cup U_2^{(8)}$.
It remains to rule out the possibility that
$u^\star\in U_2^{(8)}$.
Suppose, for contradiction, that
$u^\star\in U_2^{(8)}.$
Then, for every $u\in U_2^{(8)}$, Claim~\ref{CLA5.4} gives
$y_{u^\star}-y_u\leq16\varepsilon_2y_{u^\star},$
and hence
$y_u\geq(1-16\varepsilon_2)y_{u^\star}.$

Choose an arbitrary vertex $w\in U_1^{(8)}$.
By the definition of $U_1^{(8)}$,
$|N_{U_2}(w)|\geq(1-8\varepsilon_2)|U_2|.$
Recall that $H^\star$ is obtained from $H$ by deleting the edges in
$E'\cup E''$ and adding the matching $M$.
Since
$E'\subseteq E(H[U_1])$ and
$M\subseteq\binom{U_1}{2}$,
neither $E'$ nor $M$ affects the edges between $U_1$ and $U_2$.
Thus, the only edges between $U_1$ and $U_2$ that are deleted when
passing from $H$ to $H^\star$ are those in $E''$.
Since $E''$ is a matching, the vertex $w$ loses at most one neighbor
in $U_2$. Therefore,
$$
|N^\star_{U_2}(w)|
\geq
|N_{U_2}(w)|-1.
$$
Using $|S_2^{(8)}|\leq\varepsilon_2|U_2|$, we obtain
$$
\begin{aligned}
|N^\star_{U_2^{(8)}}(w)|
\geq
|N^\star_{U_2}(w)|-|S_2^{(8)}|
\geq
(1-9\varepsilon_2)|U_2|-1
\geq
(1-10\varepsilon_2)|U_2|,
\end{aligned}
$$
where the last inequality holds for sufficiently large $h$.

Applying the eigenvalue equation at $w$, we obtain
$$
\begin{aligned}
\rho^\star y_w
=
\sum_{v\in N_{H^\star}(w)}y_v
\geq
\sum_{v\in N^\star_{U_2^{(8)}}(w)}y_v
\geq
(1-10\varepsilon_2)|U_2|
(1-16\varepsilon_2)y_{u^\star}.
\end{aligned}
$$

Since $e(H^\star)=h$, by Lemma \ref{lem2.8} we have
$\rho^\star\leq\sqrt{2h}$.
while \eqref{ali-02} gives
$|U_2|\geq9\sqrt h.$
Consequently,
$$
y_w
\geq
\frac{9(1-10\varepsilon_2)(1-16\varepsilon_2)}{\sqrt{2}}y_{u^\star}
>y_{u^\star},
$$
where the last inequality holds since the parameters are sufficiently
small.
This contradicts the choice of $u^\star$ as a vertex maximizing the
entries of $\mathbf{y}$.
Therefore,
$u^\star\notin U_2^{(8)}.$
Since
$u^\star\in U_1^{(8)}\cup U_2^{(8)}$,
we conclude that
$u^\star\in U_1^{(8)}.$
This completes the proof.
\end{proof}

\begin{claim}\label{CLA5.7}
For every $u\in U_2$,  we have
$y_u\le \frac{2}{5}y_{u^\star}$.
\end{claim}

\begin{proof}
Fix a vertex $u\in U_2$.
Let $v_0\in U_2^{(8)}$ be a vertex such that $y_{v_0}=\min_{v\in U_2^{(8)}}y_v$.
By Claim~\ref{CLA5.4}, we have
$y_u-y_{v_0}\leq 16\varepsilon_2 y_{u^\star}.$
Since $\varepsilon_2$ is sufficiently small,
$\frac14+16\varepsilon_2\leq\frac25.$
Thus, it suffices to prove that
$y_{v_0}\leq \frac14y_{u^\star}.$
We prove this by contradiction. Suppose that
$
y_{v_0}>\frac14 y_{u^\star}.
$
By Claim~\ref{CLA5.6}, we have \(u^\star\in U_1^{(8)}\).
Consequently,
$$
|N^\star_{U_2^{(8)}}(u^\star)|
\geq |N^\star_{U_2}(u^\star)|-|S_2^{(8)}|
\geq \big((1-8\varepsilon_2)|U_2|-1\big)-\varepsilon_2 |U_2|
\geq (1-10\varepsilon_2)|U_2|.
$$
It follows that
\[
\rho^\star y_{u^\star}
\geq (1-10\varepsilon_2)|U_2|y_{v_0}
\geq (1-10\varepsilon_2)|U_2|\frac14 y_{u^\star}.
\]
Therefore,
$\rho^\star\geq \frac14(1-10\varepsilon_2)|U_2|.$
By \eqref{ali-02}, we have
$|U_2|\geq 9\sqrt{h}$,
and hence
$\rho^\star\geq \frac14(1-10\varepsilon_2)|U_2|
\geq 2\sqrt{h}.$
On the other hand,
since $e(H^\star)=h$, by Lemma \ref{lem2.8} we have
$\rho^\star\leq\sqrt{2h},$
a contradiction.
Therefore,
$y_{v_0}\leq \frac14 y_{u^\star},$
and hence
$y_u\leq \frac25 y_{u^\star}.$
\end{proof}

\begin{claim}\label{CLA5.8}
We have
$(\rho^\star)^2-\rho^\star\leq h-\frac1{5}|U_1|.$
\end{claim}

\begin{proof}
By Claim~\ref{CLA5.6}, we have $u^\star\in U_1^{(8)}$.
By the construction of $H^\star$, all edges of $H[U_1]$ incident with
$u^\star$ were deleted, while the added edge set $M$ is a matching.
Consequently,
$|N^\star_{U_1}(u^\star)|\leq 1.$
Applying the eigenvalue equation twice at $u^\star$, we obtain
$$
(\rho^\star)^2y_{u^\star}
=\sum_{v\in N_{H^\star}(u^\star)}
\sum_{z\in N_{H^\star}(v)}y_z.
$$
We split the outer sum according to
$N_{H^\star}(u^\star)=
N^\star_{U_1}(u^\star)\cup N^\star_{U_2}(u^\star)$.

\begin{itemize}
    \item Since $|N^\star_{U_1}(u^\star)|\leq1$, the eigenvalue equation gives
$$
\begin{aligned}
\sum_{v\in N^\star_{U_1}(u^\star)}
\sum_{z\in N_{H^\star}(v)}y_z
=
\rho^\star
\sum_{v\in N^\star_{U_1}(u^\star)}y_v
\leq
\rho^\star y_{u^\star}.
\end{aligned}
$$

\item We next estimate the contribution from
$v\in N^\star_{U_2}(u^\star)$.
First, since $y_z\leq y_{u^\star}$ for every $z\in U_1$, we obtain
$$
\begin{aligned}
\sum_{v\in N^\star_{U_2}(u^\star)}
\sum_{z\in N^\star_{U_1}(v)}y_z
&\leq
e_{H^\star}(U_1,U_2)y_{u^\star}.
\end{aligned}
$$
\item On the other hand, by Claim~\ref{CLA5.7},
$y_z\leq \frac25y_{u^\star}$ for every $z\in U_2.$
Hence
$$
\begin{aligned}
\sum_{v\in N^\star_{U_2}(u^\star)}
\sum_{z\in N^\star_{U_2}(v)}y_z
\leq 2e_{H^\star}(U_2)\cdot\frac25y_{u^\star}
=\frac45e_{H^\star}(U_2)y_{u^\star},
\end{aligned}
$$
where the factor $2$ accounts for the two possible orientations of each edge of $H^\star[U_2]$.
\end{itemize}

Combining the above estimates, we obtain
$$
(\rho^\star)^2y_{u^\star}
\leq
\left(
e_{H^\star}(U_1,U_2)
+\frac45e_{H^\star}(U_2)
\right)y_{u^\star}
+\rho^\star y_{u^\star}.
$$
Since
$h=e_{H^\star}(U_1)+e_{H^\star}(U_1,U_2)+e_{H^\star}(U_2),$
it follows that
$$
\begin{aligned}
(\rho^\star)^2-\rho^\star
\leq
e_{H^\star}(U_1,U_2)
+\frac45e_{H^\star}(U_2)
=
h-e_{H^\star}(U_1)-\frac15e_{H^\star}(U_2)
\leq
h-e_{H^\star}(U_1).
\end{aligned}
$$

Finally, by the construction of $H^\star$, the matching $M$ is contained
in $E(H^\star[U_1])$. Therefore,
$$
e_{H^\star}(U_1)
\geq |M|
=|E'|+|E''|
\geq |E''|.
$$
By Lemma~\ref{LEM4.8B},
$|U_1|\geq30s^2$, and hence
$
|E''|
=\lfloor\frac{|U_1|}{4}\rfloor\geq\frac15|U_1|.$
Consequently,
$
(\rho^\star)^2-\rho^\star
\leq
h-\frac15|U_1|,
$
as desired.
\end{proof}

By Claim \ref{CLA5.3}, we have
 $\rho^\star \geq g_{s-1}(h)=\frac{s-2 + \sqrt{4h - (s-1)^2 + 1}}{2}$.
Rearranging the inequality yields
$$(\rho^\star)^2 - (s - 2)\rho^\star \geq h - \frac{(s - 1)(s - 2)}{2}.$$
Combining this with Claim \ref{CLA5.8} yields
$(s-3)\rho^\star\leq \frac{(s - 1)(s - 2)}{2}-\frac{1}{5} |U_1|$.
Since \(\rho^\star\geq \sqrt h\), if \(s\ge4\), then the left-hand
side is at least \((s-3)\sqrt h\), whereas the right-hand side is bounded
above by a constant depending only on \(s\), which is impossible for
sufficiently large \(h\). Hence we must have \(s=3\). In this case, the
above inequality gives
$|U_1|\leq 5$,
contradicting $|U_1|\geq 30s^2$.
Thus, Case 1 is impossible.

\medskip

\noindent{{\bf{Case 2.}}} $|U_2|\leq 100|U_1|$.

\medskip

Since  $|U_1|\leq |U_2|\leq 100|U_1|$ and $|U_1||U_2|\leq (1+\varepsilon)h$,
it follows that $|U_1|\leq \sqrt{(1+\varepsilon)h}$ and $|U_2|\leq 100\sqrt{(1+\varepsilon)h}$.
Thus, for every $v\in V(H)$, we have
\begin{align}\label{Align-001}
d_H(v)\leq |U_1|+|U_2|\leq 101\sqrt{(1+\varepsilon)h}< \varepsilon_0 h.
\end{align}

\begin{claim}\label{CLA5.9}
Let $i\in\{1,2\}$ and put $j=3-i$.
Then the following assertions hold:

{\rm (i)} For every $u\in U_i^{(8)}$,
$x_u\geq (1-16\varepsilon_2) \sqrt{\frac{|U_j|}{2h}}.$

{\rm (ii)} For every $u\in U_i$, $x_u\leq (1+200\varepsilon_2) \sqrt{\frac{|U_j|}{2h}}.$ \end{claim}

\begin{proof}
By Lemma~\ref{Lem4.10C} (ii), for every $u\in U_i^{(8)}$,
$x_u^2\geq(1-30\varepsilon_2)\frac{|U_j|}{2h}.$
Since $\varepsilon_2$ is sufficiently small,
$\sqrt{1-30\varepsilon_2}\geq 1-16\varepsilon_2,$
and hence
$x_u\geq(1-16\varepsilon_2)\sqrt{\frac{|U_j|}{2h}}.$
This proves {\rm (i)}.

We now prove {\rm (ii)}.
Choose $u_0\in U_i^{(8)}$ such that
$x_{u_0}=\min_{v\in U_i^{(8)}}x_v.$
By Lemma~\ref{Lem4.10C} (iii) and
$|U_i^{(8)}|\geq(1-\varepsilon_2)|U_i|$, we have
$(1-\varepsilon_2)|U_i|x_{u_0}^2\leq\sum_{v\in U_i^{(8)}}x_v^2
\leq\frac12+20\varepsilon_2.$
Therefore,
$x_{u_0}^2\leq\frac{1+40\varepsilon_2}{2(1-\varepsilon_2)|U_i|}$.
Since
$|U_i||U_j|\geq(1-\varepsilon)h,$
we obtain
$x_{u_0}\leq(1+21\varepsilon_2)\sqrt{\frac{|U_j|}{2h}}.$

By Lemma~\ref{Lem4.12C}, we have
$u^*\in U_1^{(8)}\cup U_2^{(8)}$.
If $u^*\in U_i^{(8)}$, then
$x_{u^*}^2\leq(1+80\varepsilon_2)\frac{|U_j|}{2h}.$
If $u^*\in U_j^{(8)}$, then
$x_{u^*}^2\leq(1+80\varepsilon_2)\frac{|U_i|}{2h}.$
Since we are in Case 2,
$\frac1{100}\leq\frac{|U_i|}{|U_j|}\leq100.$
Thus, in either case,
$x_{u^*}\leq11\sqrt{\frac{|U_j|}{2h}}$
for sufficiently small $\varepsilon_2$.

Finally, for every $u\in U_i$, Lemma~\ref{Lem4.10C} (i) gives
$x_u-x_{u_0}\leq16\varepsilon_2x_{u^*}.$
Consequently,
\begin{align*}
x_u
\leq
x_{u_0}+16\varepsilon_2x_{u^*}
\leq
\left(1+21\varepsilon_2+176\varepsilon_2\right)
\sqrt{\frac{|U_j|}{2h}}
\leq
(1+200\varepsilon_2)
\sqrt{\frac{|U_j|}{2h}}.
\end{align*}
This proves {\rm (ii)}.
\end{proof}

By Lemma \ref{Lem4.12C}, $u^*\in U_k^{(8)}$ for some $k\in \{1,2\}$.
Furthermore, by  Claim \ref{CLA5.9},
we have
\begin{align}\label{align-005G}
x_{u^*}^2\leq (1+400\varepsilon_2)^2\frac{|U_{2}|}{2h}.
\end{align}

\begin{claim}\label{CLA5.10}
For each $i\in \{1,2\}$, define $$S_i^\bullet=\{u\in U_i:|N_{U_{3-i}}(u)|\leq 0.7|U_{3-i}|\}.$$
Then we have \(S_1^\bullet\cup S_2^\bullet=\varnothing\).
\end{claim}

\begin{proof}
Suppose to the contrary that \(S_1^\bullet\cup S_2^\bullet\neq \varnothing\).
Take a vertex $w\in S_1^\bullet\cup S_2^\bullet$.
Put $a:=|N_{U_1}(w)|/|U_1|$ and $b:=|N_{U_2}(w)|/|U_2|$.
By Lemma \ref{Lem4.9C}, we have $a\leq 6\varepsilon_2$ or $b\leq 6\varepsilon_2$.

For every \(v\in U_2\), Claim~\ref{CLA5.9} implies that
$x_v\le (1+200\varepsilon_2)
\sqrt{\frac{|U_1|}{2h}}.$
For every \(u\in U_1^{(8)}\), Claim~\ref{CLA5.9} implies that
$x_u\geq (1-16\varepsilon_2)
\sqrt{\frac{|U_2|}{2h}}.$
By the definition of $u^*$,
we have $x_{u^*}\geq x_u\ge (1-16\varepsilon_2)
\sqrt{\frac{|U_2|}{2h}}$.
Consequently,
\[
\frac{x_v}{x_{u^*}}
\leq \frac{1+200\varepsilon_2}{1-16\varepsilon_2}
r
\le
(1+\varepsilon_3)r,
\]
where $r=\sqrt{{|U_1|}/{|U_2|}}\le 1$
and the last inequality follows from
\(\varepsilon_2\ll\varepsilon_3\).
Applying the eigenvalue equation to $\mathbf{x}$ at $w$ twice, we obtain
\[
\rho^2 x_{w} = \sum_{v \in N_{H}(w)} \sum_{z \in N_{H}(v)} x_z.
\]

We now estimate the contribution of the pairs $(v, z)$ in this double sum.
First, we consider the contribution from the edges between $U_1$ and $U_2$:
\[
\sum_{v \in N_{U_2}(w)} \sum_{z \in N_{U_1}(v)} x_z
+\sum_{v \in N_{U_1}(w)} \sum_{z \in N_{U_2}(v)} x_z
\leq \Big(b+a(1+\varepsilon_3)r\Big)|U_2||U_1|x_{u^*}.
\]

It remains to estimate the contribution from the edges that do not belong to the
complete bipartite graph $K_{U_1, U_2}$.
Such edges belong to $E(H) \triangle E(K_{U_1, U_2})$, which yields
\[
\sum_{v \in N_{H}(w)} \sum_{\substack{z \in N_{H}(v) \\ vz \notin E(K_{U_1, U_2})}} x_z
\leq 2 d(H, K_{U_1, U_2}) x_{u^*}
\leq 2\varepsilon h x_{u^*},
\]
where the factor $2$ accounts for the two possible orientations of each edge.
Hence
\begin{align}\label{align-18T}
        \rho^2 x_w
        \leq \Big(a(1+\varepsilon_3)r+b+2\varepsilon\Big)
             (1+\varepsilon)hx_{u^*}.
\end{align}

If $a<1/1000$ and $b<1/1000$, then since
$r:=\sqrt{|U_1|/|U_2|}\leq1$, $\rho^2\geq h$, and the
parameters are sufficiently small, \eqref{align-18T} yields
$x_w\leq
(1+\varepsilon)
\bigl(a(1+\varepsilon_3)r+b+2\varepsilon\bigr)x_{u^*}
\leq \frac3{1000}x_{u^*}$.
Since $H$ contains no isolated vertices,
we can select a neighbor $u\in N_H(w)$.
Since $|U_2|\leq 100|U_1|$, we have
$|U_2|^2
\leq 100|U_1||U_2|
\leq 100(1+\varepsilon)h$,
and hence
$|U_2|\leq 11\sqrt{h}$,
where the last inequality holds for sufficiently small $\varepsilon$.
Combining these with $x_u\leq x_{u^*}$ and \eqref{align-005G} yields
$$x_wx_u\leq \frac{3}{1000}x_{u^*}^2\leq \frac{3}{1000}(1+400\varepsilon_2)^2\frac{|U_{2}|}{2h}\leq \frac{1}{10\sqrt{h}},$$
which contradicts Lemma \ref{Lem4.3}.

It remains to consider the case where $a \ge \frac{1}{1000}$ or $b \ge \frac{1}{1000}$.
By \eqref{Align-001}, we have $d_H(w)<\varepsilon_0 h$.
Then Lemma~\ref{Lem4.4} yields
\begin{equation}\label{eq:R-lower}
2h x_w^2 \ge (1-2\varepsilon_0)d_H(w) \geq (1-2\varepsilon_0)(a|U_1| + b|U_2|)
=(1-2\varepsilon_0)(ar^2+b)|U_2|.
\end{equation}

In light of \eqref{align-18T},
$x_w\leq (1+\varepsilon)(a(1+\varepsilon_3)r+b+2\varepsilon)x_{u^*}.$
These, together with
\(\rho^2\geq h\) and \eqref{align-005G},  give that
\[
\begin{aligned}
2h x_w^2
&\leq 2h\cdot(1+\varepsilon)^2\Big(a(1+\varepsilon_3)
r+b+2\varepsilon\Big)^2x_{u^*}^2\\
&\leq\left(a(1+\varepsilon_3)r+b+2\varepsilon\right)^2(1+401\varepsilon_2)^2|U_{2}|.
\end{aligned}
\]

By Lemma~\ref{Lem4.9C}, either
$a\le 6\varepsilon_2$ or $b\le 6\varepsilon_2$.
We distinguish the following two cases:

First suppose that
$a\leq 6\varepsilon_2$.
Then \(w\notin S_2^\bullet\), since otherwise
Lemma~\ref{Lem4.5} (i) would imply that
$a\geq \frac12$, a contradiction.
Therefore, as
\(w\in S_1^\bullet\cup S_2^\bullet\),
we must have
\(w\in S_1^\bullet\).
Then by the definition of
$S_1^\bullet$, we have
$b\leq 0.7$.
Moreover, since we are in the case
\(a\ge 1/1000\) or \(b\ge 1/1000\),
and \(a\le 6\varepsilon_2<1/1000\),
we have \(b\ge1/1000\).
Hence
\[
\begin{aligned}
\left(a(1+\varepsilon_3)r+b+2\varepsilon\right)^2
(1+401\varepsilon_2)^2
&\leq
\left(b+6\varepsilon_2(1+\varepsilon_3)+2\varepsilon\right)^2
(1+401\varepsilon_2)^2\\
&\leq
\left(b+9\varepsilon_2\right)^2
(1+401\varepsilon_2)^2 .
\end{aligned}
\]
Since \(1/1000\leq b\leq0.7\) and \(\varepsilon_2\) is sufficiently small,
a direct calculation gives
\[
(b+9\varepsilon_2)^2(1+401\varepsilon_2)^2
\leq b^2+820\varepsilon_2
\leq \frac34 b
\leq \frac34(ar^2+b).
\]

Now suppose that
$b\le 6\varepsilon_2$.
Then \(w\notin S_1^\bullet\), since otherwise
Lemma~\ref{Lem4.5} (i) would imply that
$b\geq \frac12$, a contradiction.
Therefore, as
\(w\in S_1^\bullet\cup S_2^\bullet\),
we must have
\(w\in S_2^\bullet\).
Then by the definition of
$S_2^\bullet$, we have
$a\leq 0.7$.
Moreover, since we are in the case
\(a\ge 1/1000\) or \(b\ge 1/1000\),
and \(b\le 6\varepsilon_2<1/1000\),
we have \(a\ge1/1000\).
Consequently,
\[
\begin{aligned}
\left(a(1+\varepsilon_3)r+b+2\varepsilon\right)^2
(1+401\varepsilon_2)^2
&\le
\left(a(1+\varepsilon_3)r+8\varepsilon_2\right)^2
(1+401\varepsilon_2)^2 \\
&\leq
(1+3\varepsilon_3)a^2r^2+820\varepsilon_2.
\end{aligned}
\]
Since $|U_2|\leq100|U_1|$, we have $r^2\geq1/100$.
Hence, as $a\geq1/1000$,
$ar^2\geq\frac1{100000}.$
Moreover, since $a\leq 0.7$,
$a^2r^2\leq 0.7ar^2$.
For sufficiently small $\varepsilon_3$,
$(1+3\varepsilon_3)a^2r^2\leq 0.72ar^2.$
Since $\varepsilon_2$ is sufficiently small, we also have
$820\varepsilon_2\leq 0.03ar^2.$
Now we get
$
(1+3\varepsilon_3)a^2r^2+820\varepsilon_2
\leq\frac34ar^2.
$
Therefore,
\[
\left(a(1+\varepsilon_3)r+b+2\varepsilon\right)^2
(1+401\varepsilon_2)^2
\le
\frac34 ar^2
\le
\frac34(ar^2+b).
\]

In both cases, we obtain
$(a(1+\varepsilon_3)r+b+2\varepsilon)^2
(1+401\varepsilon_2)^2
\le \frac34(ar^2+b)$,
which implies $2h x_w^2<\frac34(ar^2+b)|U_2|$.
However, this contradicts \eqref{eq:R-lower}.
Therefore, $S_1^\bullet\cup S_2^\bullet=\varnothing$.
\end{proof}

\begin{claim}\label{CLA5.11}
For each $j\in \{1,2\}$,
we have $|E^j|\leq \varepsilon_2 |U_j|$,
where $E^j:=E(U_j)$.
\end{claim}

\begin{proof}
Fix $j\in\{1,2\}$.
Suppose to the contrary that
$|E^j|>\varepsilon_2|U_j|$.
Choose an arbitrary edge $e=uv\in E^j$.
By Claim \ref{CLA5.10}, we have $S_j^\bullet=\varnothing$.
Hence
$|N_{U_{3-j}}(u)|>0.7|U_{3-j}|$
and
$|N_{U_{3-j}}(v)|>0.7|U_{3-j}|$.
It follows that
$|N_{U_{3-j}}(u)\cap N_{U_{3-j}}(v)|\geq 0.4|U_{3-j}|.$
Let $L:=N_{U_{3-j}}(u)\cap N_{U_{3-j}}(v)\cap U_{3-j}^{(8)}$.
By Lemma \ref{Lem4.6}, we have
$|S_{3-j}^{(8)}|\leq \varepsilon_2|U_{3-j}|$.
Therefore,
\[
\begin{aligned}
|L|\geq
|N_{U_{3-j}}(u)\cap N_{U_{3-j}}(v)|-|S_{3-j}^{(8)}|
\geq (0.4-\varepsilon_2)|U_{3-j}|
\geq \frac13|U_{3-j}|,
\end{aligned}
\]
where the last inequality holds for sufficiently small $\varepsilon_2$.

Since we are in Case 2, we have
$|U_1|\leq |U_2|\leq 100|U_1|$
and $|U_1||U_2|\geq (1-\varepsilon)h$.
Hence both $|U_1|$ and $|U_2|$ tend to infinity as $h\to\infty$.
We first choose a $t$-set $B\subseteq L$.
Since $t$ is fixed and $|L|\geq \frac13|U_{3-j}|$,
there exists a positive constant $\alpha_{s,t}>0$ such that
$\binom{|L|}{t}\geq\alpha_{s,t}|U_{3-j}|^t$.

For a fixed $t$-set $B\subseteq L$, put
$C_B:=\bigcap_{y\in B}N_{U_j}(y).$
Since $B\subseteq U_{3-j}^{(8)}$, for every $y\in B$ we have
$|N_{U_j}(y)|\geq(1-8\varepsilon_2)|U_j|$.
Hence, by the union bound,
$|C_B|\geq(1-8t\varepsilon_2)|U_j|$.
Since $s$ and $t$ are fixed, $\varepsilon_2$ is sufficiently small,
and $|U_j|\to\infty$, we have
$|C_B\setminus\{u,v\}|\geq \frac12|U_j|$
for sufficiently large $h$.
Consequently, there exists a positive constant
$\beta_{s,t}>0$ such that
\[
\binom{|C_B\setminus\{u,v\}|}{s-2}
\geq
\beta_{s,t}|U_j|^{s-2}.
\]

For every choice of $A'\subseteq C_B\setminus\{u,v\}$ with $|A'|=s-2$,
let $A:=A'\cup\{u,v\}$.
Since $B\subseteq N_{U_{3-j}}(u)\cap N_{U_{3-j}}(v)$
and $A'\subseteq C_B$, all edges between $A$ and $B$ are present in
$H$. Together with the edge $uv$, these edges form a copy of
$K_{s,t}^{+}$ in $H$, with $uv$ as the added edge.

As in the proof of Lemma \ref{Lem4.7B}, each copy is counted at most
a constant number of times depending only on $s$ and $t$.
Therefore, there exists a positive constant $c_{s,t}>0$ such that
$|\mathcal{F}(e)|\geq c_{s,t}|U_{3-j}|^t|U_j|^{s-2}.$
Since $t+1\geq s$, we have
$\frac{t-s+1}{2}\geq0.$
Moreover, in Case 2,
$\frac1{100}\leq\frac{|U_{3-j}|}{|U_j|}\leq100$.
Hence
\[
\begin{aligned}
|U_{3-j}|^t|U_j|^{s-2}
&=
|U_{3-j}|
\bigl(|U_j||U_{3-j}|\bigr)^{\frac{s+t-3}{2}}
\left(
\frac{|U_{3-j}|}{|U_j|}
\right)^{\frac{t-s+1}{2}}\\
&\geq
100^{-\frac{t-s+1}{2}}
(1-\varepsilon)^{\frac{s+t-3}{2}}
|U_{3-j}|h^{\frac{s+t-3}{2}}.
\end{aligned}
\]
Thus, for some positive constant $c'_{s,t,\varepsilon_2}$,
$
|\mathcal{F}(e)|
\geq
c'_{s,t,\varepsilon_2}
|U_{3-j}|
h^{\frac{s+t-3}{2}}.
$

Furthermore, the families $\mathcal{F}(e)$, $e\in E^j$, are pairwise
disjoint, since in every copy of $K_{s,t}^{+}$ the added edge is the
unique edge whose deletion yields a copy of $K_{s,t}$.
Therefore, if $|E^j|>\varepsilon_2|U_j|$, then
\[
\begin{aligned}
N(K_{s,t}^+,H)
\geq
\sum_{e\in E^j}|\mathcal{F}(e)|
>\varepsilon_2|U_j|
\cdot
c'_{s,t,\varepsilon_2}
|U_{3-j}|
h^{\frac{s+t-3}{2}}
\geq
2^{\frac{s+t-1}{2}}
\varepsilon_1
h^{\frac{s+t-1}{2}},
\end{aligned}
\]
where the last inequality follows from
$|U_j||U_{3-j}|\geq(1-\varepsilon)h$
and $\varepsilon_1\ll\varepsilon_2$.
This contradicts \eqref{align-002T}.
Hence, $|E^j|\leq\varepsilon_2|U_j|$.
\end{proof}

For each $i\in \{1,2\}$, set $U_i^\bullet=U_i\setminus S_i^\bullet$.
By Lemma \ref{Lem4.11C} and  Claim \ref{CLA5.10}, we know that $R\cup S_1^\bullet\cup S_2^\bullet=\varnothing$,
which implies that $V(H)=U_1^\bullet\cup U_2^\bullet$.
Let $H^\bullet$ be the graph obtained from $H$ by deleting all edges in $E^1\cup E^2$.
For brevity, we write $\rho^\bullet = \rho(H^\bullet)$.
Then $H^\bullet$ is a subgraph of $K_{U_1,U_2}$,
and hence it is bipartite.
Furthermore, applying Lemma \ref{lem2.8} (ii) to $H^\bullet$ yields $\rho^\bullet\leq \sqrt{e(H^\bullet)}\leq \sqrt{h}$.

In the following, we will prove that
 $\rho^\bullet>\sqrt{h}$.
For each edge $uv\in E^i$, we have $u,v\in U_i$.
Therefore, Claim~\ref{CLA5.9} (ii) gives
$x_ux_v
\leq(1+400\varepsilon_2)^2\frac{|U_{3-i}|}{2h}.$
Since $|E^i|\leq\varepsilon_2|U_i|$ and
$|U_1||U_2|\leq(1+\varepsilon)h$, we have
\[
|E^i||U_{3-i}|
\leq
\varepsilon_2|U_i||U_{3-i}|
=
\varepsilon_2|U_1||U_2|
\leq
(1+\varepsilon)\varepsilon_2h
\]
for each $i\in\{1,2\}$.
By the Rayleigh principle,
\begin{align*}
\rho^\bullet-\rho
&\geq
\mathbf{x}^{\top}
\bigl(A(H^\bullet)-A(H)\bigr)\mathbf{x}
\geq
-2\left(
\sum_{uv\in E^1}x_ux_v+\sum_{uv\in E^2}x_ux_v\right)\\
&\geq
-2\sum_{i=1}^{2}
|E^i|\cdot (1+400\varepsilon_2)^2\frac{|U_{3-i}|}{2h}\\
&\geq  -4(1+400\varepsilon_2)^2\cdot \frac12(1+\varepsilon)\varepsilon_2.
\end{align*}
Since  $\varepsilon\ll \varepsilon_2\ll  1$,
it follows that
$\rho^\bullet>\rho-3\varepsilon_2$.
Furthermore, by Lemma \ref{Lem4.2}, we have
$\rho^\bullet> \frac{s-2 + \sqrt{4h - (s-1)^2 + 1}}{2}-3\varepsilon_2.$
Using \eqref{equ-001} and $s\geq 3$, we have
$\rho^\bullet> \sqrt{h} + ( \frac{s}{2} - 1 - 4\varepsilon_2)>\sqrt{h}$,
which contradicts $\rho^\bullet\leq \sqrt{h}$.
Therefore, Case 2 is impossible.

Since both Case 1 and Case 2 are impossible, the desired conclusion follows.
This completes the proof of Theorem~\ref{thm1.1}.
\end{proof}

\section{Proofs of Theorem \ref{thm1.2} and Corollary \ref{cor1.1}}\label{sec6}

\begin{proof}[\textbf{Proof of Theorem \ref{thm1.2}}]
Applying Theorem~\ref{thm1.1} with $s=k+1$ and $t=k$ yields $N(K_{k+1,k}^{+}, G) = \Omega(m^k)$ for any $m$-edge graph $G$ satisfying
$\rho(G) > g_k(m)$.

Since $C_{2k+1}$ is a spanning subgraph of $K_{k+1,k}^{+}$, every copy of $K_{k+1,k}^{+}$ contains a copy of $C_{2k+1}$ on the same vertex set.
On the other hand, a fixed copy $C$ of $C_{2k+1}$ can be contained in at most
$\binom{2k+1}{k+1}\binom{k+1}{2}$ copies of $K_{k+1,k}^{+}$.
Indeed, since both graphs have exactly $2k+1$ vertices, any copy of $K_{k+1,k}^{+}$ containing $C$ must share the same vertex set as $C$.
On this fixed vertex set, a copy of $K_{k+1,k}^{+}$ is uniquely determined by partitioning the vertices into two parts of sizes $k+1$ and $k$, and choosing the extra edge within the larger part.
Thus, the number of such copies is at most
$\binom{2k+1}{k+1}\binom{k+1}{2}$.
Therefore, we have
$$N(K_{k+1,k}^{+}, G) \le \binom{2k+1}{k+1}\binom{k+1}{2} N(C_{2k+1}, G).$$
Given that $N(K_{k+1,k}^{+}, G) = \Omega(m^k)$, it immediately follows that $N(C_{2k+1}, G) = \Omega(m^k)$.

It remains to show that the bound $m^k$ is tight.
For a sufficiently large $m$, we choose integers $r$ and $q$ such that
$m = (k+1)r + \binom{k+1}{2} + q$, where $0 \le q < k+1$.
Let $H_m$ be the graph obtained from $K_{k+1} \vee \overline{K_r}$ by adding $q$ pendant edges incident to a fixed vertex of the clique $K_{k+1}$.
Clearly, $e(H_m) = m$.
Let $Q:= K_{k+1} \vee \overline{K_r}$.
The partition of the vertex set of $Q$ into the clique part and the independent part is equitable, and its corresponding quotient matrix is
\[
    \begin{pmatrix}
    k & r\\
    k+1 & 0
    \end{pmatrix}.
\]
Thus, the spectral radius of $Q$ is
$\rho(Q) = \frac{k + \sqrt{k^2 + 4(k+1)r}}{2}.$
Since $Q \subseteq H_m$, by the monotonicity of the spectral radius,
we have $\rho(H_m) \ge \rho(Q)$.
By substituting the expression for $m$ into the term $\sqrt{4m-k^2+1}$, we obtain
\begin{align*}
    \rho(Q) - \frac{k-1+\sqrt{4m-k^2+1}}{2}
    = \frac{1}{2} - o(1)
\end{align*}
as $r \to \infty$ (or equivalently, $m \to \infty$), since $k$ is fixed and $0 \le q < k+1$.
Therefore, for all sufficiently large $m$, we have
$\rho(H_m) \ge \rho(Q) > \frac{k-1+\sqrt{4m-k^2+1}}{2}.$

Finally, we estimate the number of copies of \(C_{2k+1}\) in \(H_m\). No pendant
vertex lies on a cycle, so every copy of \(C_{2k+1}\) is contained in
$Q=K_{k+1}\vee \overline{K_r}$.
In the graph \(Q\), the independent part has no internal edges.
Therefore, on any cycle, the number of vertices taken from the independent
part is at most the number of vertices taken from the clique part.
Since the clique
part has only \(k+1\) vertices and the cycle has length \(2k+1\), every copy of
\(C_{2k+1}\) uses exactly \(k\) vertices from the independent part and all \(k+1\)
vertices from the clique part.
Since the $k+1$ vertices of the clique are fixed, the number of such vertex sets is precisely $\binom{r}{k}$. For each chosen vertex set, the number of cycles of length $2k+1$ is a constant depending only on $k$. Consequently, since $r = \Theta(m)$, we have
$$N(C_{2k+1}, H_m) = \Theta_k(r^k) = \Theta_k(m^k).$$
This construction satisfies the required spectral condition while containing only $\Theta(m^k)$ copies of $C_{2k+1}$.
Hence the lower bound
\(\Omega(m^k)\) is tight up to a constant factor.
\end{proof}

\begin{proof}[\textbf{Proof of Corollary~\ref{cor1.1}}]
Applying Theorem~\ref{thm1.1} with $s=k+1$ and $t=k$, we obtain
$$
N(K_{k+1,k}^{+},G)=\Omega(m^k)
$$
for every $m$-edge graph $G$ satisfying
$\rho(G)>g_k(m).$
Since $F\subseteq K_{k+1,k}^{+}$,
every copy of $K_{k+1,k}^{+}$ in $G$ contains a copy of $F$.
Furthermore, because $F$ and $K_{k+1,k}^{+}$ have the same number of vertices, each copy of $F$ is contained in at most a constant number of copies of $K_{k+1,k}^{+}$, where the constant depends only on $F$. It follows that
$$
N(F,G)=\Omega\bigl(N(K_{k+1,k}^{+},G)\bigr)=\Omega(m^k).
$$

It remains to show that the bound is tight.
Since $C_{2k+1}\subseteq F$ and the two graphs have the same number of vertices,
each copy of $C_{2k+1}$ in $G$ can be extended to at most a constant number of copies of $F$. Consequently,
$N(F,G)=O\bigl(N(C_{2k+1},G)\bigr)$.
By the tightness assertion in Theorem~\ref{thm1.2}, there exist graphs $G$ satisfying $\rho(G)>g_k(m)$ and
$N(C_{2k+1},G)=\Theta(m^k)$.
For these graphs, we therefore have
$$
N(F,G)=O(m^k).
$$
Hence, the bound is tight up to a constant factor.
\end{proof}

\end{document}